\documentclass[12pt]{amsart}
\usepackage{amsmath, amsthm, mathrsfs, amssymb, verbatim, bbm, url, appendix}
\usepackage{amscd}
\usepackage{latexsym,amssymb}
\usepackage{graphicx}
\usepackage{color}
\usepackage{enumerate}
\allowdisplaybreaks[3]

\usepackage{latexsym}
\usepackage{epsfig}
\usepackage{amsfonts}
\usepackage{enumerate}
\usepackage{times}

\newtheorem{theorem}{Theorem}[section]
\newtheorem{deftang}[theorem]{Definition}
\newtheorem{newrem}[theorem]{Remark}
\newtheorem{correct}[theorem]{Remark}
\newtheorem{corollary}[theorem]{Corollary}

\newtheorem{assump}[theorem]{Assumption}

\newtheorem{definition}[theorem]{Definition}

\theoremstyle{remark}

\theoremstyle{remark}
\newtheorem{notrem}[subsection]{Notational Remark}

\newtheorem{remark}{Remark}[section]

\numberwithin{equation}{section}

\numberwithin{equation}{section}

\newcommand{\ol}{\overline{L}}

\newcommand{\inst}{\overline{I\mathcal{X}}}
\newcommand{\ops}{\overline{\psi}}
\newcommand{\ra}{\right\rangle}
\newcommand{\la}{\left\langle}
\newcommand{\ix}{\mathcal{X}}
\newcommand{\iy}{\mathcal{Y}}

\newcommand{\ovkkdcc}{\overline{\mathcal{K}}_{g,n, \vec{\mu}}(\ix,d)}
\newcommand{\ovkkd}{\overline{\mathcal{K}}_{g,n}(\ix,d)}

\newcommand{\ovkd}{\overline{\mathcal{K}}_{0,n}(\ix,d)}

\newtheorem{d1sec2}{Definition}[section]
\newtheorem{d1sec3}{Theorem}[section]
\newtheorem{assump1}[d1sec2]{Assumption}

\newtheorem{p1sec3}[d1sec3]{Proposition}
\newtheorem{d2sec2}[d1sec2]{Definition}

\newtheorem{th1sec4}{Theorem}[section]
\newtheorem{th1sec2}[d1sec2]{Theorem}
\newtheorem{p1sec5}{Proposition}[section]
\newtheorem{p3sec5}[p1sec5]{Proposition}
\newtheorem{p4sec5}[p1sec5]{Proposition}
\newtheorem{p5sec5}[p1sec5]{Proposition}
\newtheorem{p6sec5}[p1sec5]{Proposition}
\newtheorem{p7sec5}[p1sec5]{Proposition}
\newtheorem{l1sec5}[p1sec5]{Lemma}
\newtheorem{l2sec5}[p1sec5]{Lemma}

\newtheorem{p8sec5}[p1sec5]{Proposition}
\newtheorem{c2sec7}{Corollary}[section]

\newtheorem{tha02}{Theorem}[section]

\newcommand{\K}{\mathcal{K}}

\def\<{\left\langle}
\def\>{\right\rangle}

\begin{document}

\title[Genus $0$ quantum orbifold Hirzebruch Riemann-Roch]{Quantum orbifold Hirzebruch-Riemann-Roch theorem\\ in genus zero}
\author[Tonita]{Valentin Tonita}
\address{Institute  de math\'ematiques Jussieu\\ Universit\'e de Paris Rive-Gauche  \\ 4 Place Jussieu\\ Paris 75005\\ France}
\email{valentin.tonita@imj-prg.fr}

\author[Tseng]{Hsian-Hua Tseng}
\address{Department of Mathematics\\ Ohio State University\\ 100 Math Tower, 231 West 18th Ave. \\ Columbus \\ OH 43210\\ USA}
\email{hhtseng@math.ohio-state.edu}

\date{\today}

\begin{abstract}
We introduce K-theoretic Gromov-Witten invariants of algebraic orbifold target spaces. Using the methods developed in \cite{gito} we characterize Givental's Lagrangian cone of quantum K-theory of orbifolds in terms of the cohomological cone.
\end{abstract}

\maketitle

\section{Introduction}
K-theoretic Gromov-Witten invariants, introduced by Givental \cite{givental_wdvv} and Y.-P. Lee \cite{ypl}, are holomorphic Euler characteristics of certain bundles on the moduli spaces of stable maps to a complex projective manifold. These invariants are known to satisfy certain finite-difference equations (see e.g. \cite{gito}). K-theoretic Gromov-Witten invariants of homogeneous spaces have rich connections to combinatorics and representation theory. They are related to integrable systems (see e.g. \cite{giv_lee} and \cite{B_F}). More recently, the relationship between K-theoretic and cohomological Gromov-Witten invariants in genus $0$ have been obtained, see \cite{gito}. 

In this paper we generalize the definition of K-theoretic Gromov-Witten invariants to the case when the target is a smooth projective Deligne-Mumford $\mathbb{C}$-stack $\ix$. This is done in Section \ref{subsec:K_inv}. Just like the manifold case, genus $0$ invariants can be cast into a loop space formalism. In particular this means that the totality of genus $0$ K-theoretic Gromov-Witten invariants can be encoded in a Lagrangian cone $\mathcal{L}^\mathcal{K}$. We explain how to do this in Section \ref{subsec:loop_space}.

The main result of this paper is Theorem \ref{main1}. The statement of Theorem \ref{main1} is highly technical and is explained in more details in Section \ref{sec:main_thm}. Roughly speaking, Theorem \ref{main1} characterizes points on the cone $\mathcal{L}^\mathcal{K}$ in terms of fake K-theoretic Gromov-Witten invariants (as defined in Section \ref{sec:fake_GW}) and twisted orbifold Gromov-Witten invariants (\cite{tseng}, \cite{to1}). Since both twisted theory and fake theory are expressible in terms of the usual cohomological Gromov-Witten theory (\cite{tseng}, \cite{to1}), Theorem \ref{main1} expresses K-theoretic orbifold Gromov-Witten theory in terms of the cohomological ones. Hence Theorem \ref{main1} may be called ``quantum orbifold Hirzebruch-Riemann-Roch theorem''. 

A main geometric ingredient in the proof of Theorem \ref{main1} is the virtual Kawasaki Riemann-Roch formula, which calculates Euler characteristics in the presence of virtual structure sheaves, see Section \ref{subsec:KRR}. Processing contributions in virtual Kawasaki Riemann-Roch formula applied to moduli spaces of orbifold stable maps is the main technical aspect of the proof of Theorem \ref{main1}, which occupies Section \ref{sec:proof_main}. We follow the approach in \cite{gito}.  
   
\subsection*{Acknowledgment}
We thank D. Abramovich, A. Givental, and Y. Ruan for useful discussions. We also thank H. Iritani and T. Milanov for pointing out some errors in an earlier version of the paper.
 
Part of this work was carried out during a visit of H.-H. T. to Kavli IPMU. It is a pleasure to acknowledge their hospitality and support.
 
V. T. acknowledges the World Premiere International Research Center Initiative (WPI Initiative), Mext, Japan. H.-H. T. is supported in part by Simons Foundation Collaboration Grant.   

\section{Preparatory materials}   
\subsection{K-theoretic orbifold Gromov-Witten theory}\label{subsec:K_inv}
The purpose of this subsection is to explain the definition of K-theoretic Gromov-Witten invariants of orbifolds. We first briefly recall the notion of orbifold stable maps to an orbifold $\ix$, as introduced in \cite{abvi}.
\begin{d1sec2}[orbicurves]
   A nodal $n$-pointed orbicurve $(\mathcal{C}, \Sigma_1 ,\Sigma_2, \ldots , \Sigma_n)$ is a nodal marked complex curve such that  
     \end{d1sec2}   
      \begin{itemize}
    \item $\mathcal{C}$ has trivial orbifold structure on the complement of the marked points and nodes.
    \item  Locally near a marked point,  $\mathcal{C}$ is isomorphic to $[$Spec $\mathbb{C}[z]/\mathbb{Z}_r]$, for some $r$, and the generator of $\mathbb{Z}_r$ acts by $z\mapsto \zeta z$, $\zeta^r=1$. 
    \item  Locally near a node,  $\mathcal{C}$ is isomorphic to  $[$Spec $\left(\mathbb{C}[z,w]/(zw)\right) /\mathbb{Z}_r]$,  and the generator of $\mathbb{Z}_r$ acts by $z\mapsto \zeta z$, $w\mapsto \zeta^{-1}w$. We call this action {\em balanced} at the node.
    \end{itemize} 
    \begin{d2sec2}[orbifold stable maps]\label{def002}
    An  $n$-pointed, genus $g$, degree $d$ orbifold stable map is a representable morphism $f: \mathcal{C}\to \ix$ , whose domain is an $n$-pointed genus $g$ orbicurve $\mathcal{C}$ such that $f_*([\mathcal{C}])=d\in H_2(\ix,\mathbb{Q})$.  
    \end{d2sec2}
  We denote the moduli spaces\footnote{In \cite{abgrvi2}, the notation for this moduli space is $\mathcal{K}$. We choose to use $\overline{\mathcal{K}}$ in order to avoid confusion with the symplectic vector space defined in Section \ref{subsec:loop_space}.} of twisted stable maps by $\ovkkd$. More detailed discussions about this moduli space can be found in \cite{abgrvi2}.  
  
    Let $$I\mathcal{X}:=\mathcal{X}\times_{\Delta, \mathcal{X}\times \mathcal{X}, \Delta} \mathcal{X}$$
    be the inertia stack of $\mathcal{X}$, where $\Delta: \mathcal{X}\to \mathcal{X}\times \mathcal{X}$ is the diagonal morphism. The objects of $I\mathcal{X}$ may be described as follows: $$Ob(I\mathcal{X})=\{(x,g) | x\in Ob(\mathcal{X}), g\in Aut_\mathcal{X}(x)\}.$$ We write $$I\mathcal{X}=\coprod_{\mu\in \mathcal{I}} \mathcal{X}_\mu$$ for the decomposition of $I\mathcal{X}$ as a disjoint union of connected components. Here $\mathcal{I}$ is a index set. There is a distinguished component $$\mathcal{X}_0:=\{(x,id) | x\in Ob(\mathcal{X}), id\in Aut_\mathcal{X}(x) \text{ is the identity element}\}\subset I\mathcal{X}$$ which is canonically isomorphic to $\mathcal{X}$. It is often called the {\em untwisted sector}. 
    
There is a natural involution $$\iota: I\mathcal{X}\to I\mathcal{X}, \quad (x,g)\mapsto (x, g^{-1}).$$
    
We consider the {\em rigidified inertia stack} $$\overline{I\mathcal{X}}=\coprod_{\mu\in \mathcal{I}} \overline{\mathcal{X}}_\mu,$$
see \cite{abgrvi2} for more details. We write $\overline{\iota}: \overline{I\mathcal{X}}\to \overline{I\mathcal{X}}$ for the involution induced by $\iota$.  

There is a locally constant function $\text{age}: \overline{I\ix}\to \mathbb{Q}$. We write $\text{age}(\overline{\ix}_\mu)$ for its value on the component $\overline{\ix}_\mu$. See \cite{chru2}, \cite{abgrvi2} for the definition.
    
 \begin{notrem}
 In what follows we consider various quantities taking values in cohomology $H^*(\overline{I\ix})=\oplus_\mu H^*(\ix_\mu)$. For a quantity $\mathbf{F}$ taking values in $H^*(\overline{I\ix})$ we denote by $\mathbf{F}_\mu$ the projection of $\mathbf{F}$ to $H^*(\ix_\mu)$.
\end{notrem}   
    
    We denote by $ev_i: \overline{\mathcal{K}}_{g,n}(\mathcal{X}, d)\to \overline{I\mathcal{X}}$ the evaluation map at the $i$-th marked point. The connected components of $\ovkkd$ are denoted $\ovkkdcc$ where $\vec{\mu}:=(\mu_1, \ldots, \mu_{n})$ keeps track of the target of evaluation maps, i.e.
        \begin{displaymath}
    \ovkkdcc:= ev_1^{-1}(\overline{\ix}_{\mu_1})\cap \cdots \cap ev_n^{-1}(\overline{\ix}_{\mu_n})\subset \ovkkd  .
        \end{displaymath} 
    Let $$\overline{\mathcal{K}}_{g,n+1}(\mathcal{X}, d)':=ev_{n+1}^{-1}(\mathcal{X}_0)\subset \overline{\mathcal{K}}_{g,n+1}(\mathcal{X}, d).$$ 
    The stack $\overline{\mathcal{K}}_{g,n+1}(\mathcal{X}, d)'$ parametrizes $(n+1)$-pointed orbifold stable maps of given topological type such that the $(n+1)$-st marked point is non-stacky. The forgetful map 
    $$\pi: \overline{\mathcal{K}}_{g,n+1}(\mathcal{X}, d)'\to \overline{\mathcal{K}}_{g,n}(\mathcal{X}, d)$$
    which forgets the $(n+1)$-st marked point exhibits  $\overline{\mathcal{K}}_{g,n+1}(\mathcal{X}, d)'$ as the universal family over orbicurves over $\overline{\mathcal{K}}_{g,n}(\mathcal{X}, d)$, and $ev_{n+1}: \overline{\mathcal{K}}_{g,n+1}(\mathcal{X}, d)'\to \mathcal{X}$ is the universal stable map.  
    
As pointed out in \cite{abgrvi2}, $\overline{\mathcal{K}}_{g,n}(\mathcal{X}, d)$ admits a perfect obstruction theory relative to the Artin stack of prestable pointed orbicurves, given by $R\pi_*ev_{n+1}^*T_\mathcal{X}$. The general construction of \cite{behfan} yields a virtual fundamental class $$[\overline{\mathcal{K}}_{g,n}(\ix,d)]^{vir}\in H_*(\overline{\mathcal{K}}_{g,n}(\mathcal{X}, d), \mathbb{Q}).$$ The general construction in \cite[Section 2]{ypl} yields a virtual structure sheaf $$\mathcal{O}^{vir}_{\overline{\mathcal{K}}_{g,n}(\mathcal{X}, d)}\in K_0(\overline{\mathcal{K}}_{g,n}(\mathcal{X}, d)).$$
    
\begin{remark}
Let $\overline{\mathcal{M}}_{g,n}(\mathcal{X},d)$ be the moduli stack of orbifold stable maps with given topological type {\em with sections to all marked gerbes}, see \cite{abgrvi} and \cite{tseng} for more discussions on this moduli space. According to \cite{abgrvi}, $\overline{\mathcal{M}}_{g,n}(\mathcal{X},d)$ also admits a perfect relative obstruction theory and consequently a virtual structure sheaf $\mathcal{O}^{vir}_{\overline{\mathcal{M}}_{g,n}(\mathcal{X}, d)}\in K_0(\overline{\mathcal{M}}_{g,n}(\mathcal{X}, d)).$ As pointed out in \cite[Section 4.5]{abgrvi}, $\overline{\mathcal{M}}_{g,n}(\mathcal{X}, d)$ is the fiber product of universal marked gerbes over $\overline{\mathcal{K}}_{g,n}(\mathcal{X},d)$. Since the fiber product map $p: \overline{\mathcal{M}}_{g,n}(\mathcal{X}, d)\to \overline{\mathcal{K}}_{g,n}(\mathcal{X},d)$ is \'etale, \cite[Proposition 3]{ypl} implies that $p^* \mathcal{O}^{vir}_{\overline{\mathcal{K}}_{g,n}(\mathcal{X}, d)}= \mathcal{O}^{vir}_{\overline{\mathcal{M}}_{g,n}(\mathcal{X}, d)}$.
\end{remark}

For $1\leq i\leq n$, there is the $i$-th tautological cotangent line bundle $L_i\to \overline{\mathcal{K}}_{g,n}(\mathcal{X}, d)$, whose fiber at $[(\mathcal{C}, p_1,...,p_n)\to \mathcal{X}]$ is the cotangent line $T^*_{\bar{p}_i}C$ of the {\em coarse curve} $C$. We write $\ops_i=c_1(L_i)$.
    
Let $K^0(\inst)$ be the Grothendieck group of topological vector bundles on $\inst$. For $E_1,...,E_n\in K^0(\inst)$ and nonnegative integers $k_1,...,k_n$, define 
\begin{equation}
\langle E_1L^{k_1},..., E_nL^{k_n}\rangle_{g,n,d}:=\chi\left(\overline{\mathcal{K}}_{g,n}(\mathcal{X}, d), \mathcal{O}^{vir}_{\overline{\mathcal{K}}_{g,n}(\mathcal{X}, d)}\otimes \bigotimes_{i=1}^n(L_i^{\otimes k_i}\otimes ev_i^*E_i) \right). 
\end{equation}
These are the {\em K-theoretic Gromov-Witten invariants of $\mathcal{X}$}. It turns out they satisfy the same relations (string, dilaton, WDVV) as for the case of manifold target spaces. We include these relations in Appendix \ref{append_taut_eqns}.
     
Cohomological Gromov-Witten invariants of $\ix$ are defined\footnote{We use the same notation for K-theoretic and cohomological invariants. Thanks to the different notations for descendants, there should be no confusion.} as follows. For classes $\varphi_1,...,\varphi_n\in H^*(I\ix)$, define
  \begin{align}\label{eqn:GW_inv}
 \la \varphi_1 \ops^{k_1},\ldots \varphi_n \ops^{k_n} \ra_{g,n,d}:=\int_{[\overline{\mathcal{K}}_{g,n}(\ix,d)]^{vir}} \prod^n_{i=1}ev_i^* \varphi_i \ops_i^{k_i}.
 \end{align}

\subsection{Loop space formalism}\label{subsec:loop_space}     
   In this paper we focus our attention on the genus $0$ invariants. Let $\mathbf{t}(q)$ be a Laurent polynomial in $q$ with vector coefficients in $K^0(\inst)$. The genus $0$ K-theoretic Gromov-Witten potential is defined as:
 \begin{align}\label{eqn:K_potential}
 \mathcal{F}_0(\mathbf{t}(q)):= \sum_{n,d}\frac{Q^d}{n!}\la \mathbf{t}(L), \ldots ,\mathbf{t}(L)\ra_{0,n,d}.
 \end{align}
 where the sum is defined over all stable pairs $(n,d)$ and $Q^d$ is the representative of $d\in H_2(\ix,\mathbb{Q})$ in the Novikov ring which is a completion of the semigroup of effective curves in $\ix$.
The $J$-function is defined as:
\begin{align}\label{eqn:J_K}
\mathcal{J}_\ix (q,\mathbf{t}(q)) = 1-q + \mathbf{t}(q)+ \sum_{n,d,a}\frac{Q^d}{n!}\Phi_a\la \frac{\Phi^a}{1-qL}, \mathbf{t}(L), \ldots ,\mathbf{t}(L)\ra_{0,n+1,d}
\end{align}
where $$\{\Phi_a\}, \{\Phi^a\}\subset K^0(\inst)$$ are dual bases of $K^0(\inst)$ with respect to the pairing $(-,-)$ on $K^0(\overline{I\mathcal{X}})$ given by $$(A,B):=\chi(\overline{I\mathcal{X}}, A\otimes \overline{\iota}^*B).$$ 
We make the following 
\begin{assump1}
The pairing $(-,-)$ is non-degenerate.
\end{assump1}
In concrete examples, e.g. toric stacks, it is not hard to check this Assumption directly. It is not clear in what generality this Assumption holds true. We do not know of examples for which this Assumption fails. 

 We now explain how to describe genus $0$ K-theoretic Gromov-Witten theory of orbifolds in Givental's formalism \cite{given1}. The description is parallel to the manifold case discussed in \cite{gito}. 
 
 The $J$-function is viewed as a function $\mathcal{K}_+ \to \mathcal{K}$, where $\mathcal{K}$ is the K-theoretic loop space which we now introduce:
 \begin{align*}
 \mathcal{K} = \mathbb{C}(q, q^{-1}) \otimes K^0(\inst, \mathbb{C}[[Q]])  .
 \end{align*}
 
 One can endow  $\mathcal{K}$ with a symplectic structure: 
 \begin{align*}
 \Omega( \mathbf{f}(q),\mathbf{g}(q)):= [Res_{q=0}+Res_{q=\infty}]( \mathbf{f}(q), \mathbf{g}(q^{-1}))\frac{dq}{q}.
 \end{align*}
 and consider the following subspaces:
\begin{align*}
&\mathcal{K}_+:=\mathbb{C}[q, q^{-1}] \otimes K^0(\inst, \mathbb{C}[[Q]]);\\
&\mathcal{K}_- := \{\mathbf{f}\in \mathcal{K}, \mathbf{f}(0)\neq \infty , \mathbf{f}(\infty)=0\}.
\end{align*}
Then $$\mathcal{K}:=\mathcal{K}_+\oplus\mathcal{K}_-$$  is a polarization and the $J$-function is the differential of the potential regarded as a function of $1-q+\mathbf{t}(q)$. Here $1-q$ is the dilaton shift in the K-theoretic set-up. The proof is exactly as the one in \cite{gito} for manifold target spaces. 

\begin{th1sec2}\label{thm:cone_K}
Let the range of $\mathcal{J}$ be denoted by $\mathcal{L}^\mathcal{K}=\mathcal{L}\subset \mathcal{K}$. Then $\mathcal{L}$ is (the formal germ of) an overruled Lagrangian cone i.e. tangent spaces $T$ to $\mathcal{L}$ are tangent to $\mathcal{L}$ along $(1-q)T$.
\end{th1sec2}

\begin{proof}
This is based on the comparison between ancestors and descendant potentials. We include the computation in Appendix B. Then one can prove the corresponding properties for the ancestor cone $\mathcal{L}_\tau$ - this is explained in \cite{gito}. 
\end{proof}

The loop space formalism of cohomological Gromov-Witten theory of an orbifold is described in \cite{tseng}. We give a brief review. The loop space of the cohomological theory is:
  \begin{align*}
  \mathcal{H} :=H^*(\inst,\mathbb{C}[[Q]])[z^{-1},z]].
  \end{align*}
 It is endowed with the symplectic structure:
  \begin{align*}
  \Omega_H (\mathbf{f}(z),\mathbf{g}(z)) = - Res_{z=0}(\mathbf{f}(z),\mathbf{g}(-z) )
  \end{align*}
  where the pairing $(-,-)$ is the orbifold Poincar\'e pairing. The cohomological potential and $J$-function are defined using
  $\ops_i =c_1(L_i)$ in the correlators (\ref{eqn:GW_inv}):
  \begin{align*}
 \mathcal{J}^H(z,\mathbf{t}(z))= -z+\mathbf{t}(z) + \sum_{d,n,a}\varphi_a\frac{Q^d}{n!}\la \frac{\varphi^a}{-z-\ops},\mathbf{t}(\ops),\ldots ,\mathbf{t}(\ops) \ra_{0,n+1,d}.
  \end{align*}
Here $\{ \varphi_a \}$ and $\{\varphi^a\}$ are basis dual to each other with respect to orbifold Poincar\'e pairing. The range of the cohomological $J$-function is denote by $$\mathcal{L}^H\subset \mathcal{H}.$$

\subsection{Kawasaki Riemann-Roch}\label{subsec:KRR}
In this subsection we describe Kawasaki's Riemann-Roch formula. 
   \begin{d1sec3}
   Let $\iy$ be a compact complex orbifold (or smooth projective Deligne-Mumford $\mathbb{C}$-stack) and $V$ a vector orbibundle on $\iy$. Then:
   \begin{align}\label{eqn:KRR}
    \chi(\iy, V) = \sum_\mu \frac{1}{m_\mu}\int_{\iy_\mu} Td(T_{\iy_\mu})ch \left(\frac{Tr(V)}{Tr(\Lambda^\bullet N^\vee_\mu)}\right).
    \end{align}
   \end{d1sec3}
 In what follows we explain this ingredients of this formula. $I\iy$ is the inertia orbifold of $\iy$, which may be described as follows: around any point $p\in \iy$ there is a local chart $(\widetilde{U}_p, G_p)$ such that locally $\iy$ is represented as the quotient of $\widetilde{U}_p$ by $G_p$. Consider the set of conjugacy classes $(1)=(h_p^1)$, $(h_p^2)$, $\ldots$, $(h_p^{n_p})$ in $G_p$. Then
      \begin{align*}
      I\iy=\{\left(p,(h_p^i)\right) \quad \vert \quad i=1,2,\ldots, n_p \}.
       \end{align*}
       
   Pick an element $h_p^i$ in each conjugacy class. Then a local chart on $I\iy$ is given by:
    \begin{align*}
       \coprod_{i=1}^{n_p} [\widetilde{U}_p^{(h_p^i)}/ Z_{G_p}(h_p^i)],
     \end{align*}
     where $Z_{G_p}(h_p^i)$ is the centralizer of $h_p^i$ in $G_p$.
   Denote by $\iy_\mu$ the connected components\footnote{We also refer to them as Kawasaki strata.} of the inertia orbifold $I\iy$. The multiplicity $m_\mu$ associated to each $\iy_\mu$ is given by
      \begin{align*}
     m_\mu:= \left\vert ker\left(Z_{G_p}(g)\rightarrow Aut(\widetilde{U}_p^g)\right) \right\vert . 
     \end{align*} 
   
   For a vector bundle $V$ we will denote by $V^\vee$ the dual bundle to $V$. The restriction of $V$ to $\iy_\mu$ decomposes in characters of the $g$ action. Let $V_r^{(l)}$ be the subbundle of the restriction of $V$ to $\iy_\mu$ on which $g$ acts with eigenvalue $e^{\frac{2\pi i l}{r}}$. Then the  trace $Tr(V)$ is defined to be the orbibundle whose fiber over the point $(p, (g))$ of $\iy_\mu$ is 
    \begin{align*}
 Tr(V):= \sum_{0\leq l\leq r-1} e^{\frac{2\pi i l}{r}} V^{(l)}_r . 
    \end{align*} 
 Finally, $\Lambda^\bullet N^\vee_\mu$ is the K-theoretic Euler class of the normal bundle $N_\mu$ of $\iy_\mu$ in $\ix$. $Tr(\Lambda^\bullet N^\vee_\mu)$ is invertible because the symmetry $g$ acts with eigenvalues different from $1$ on the normal bundle to the fixed point locus.   
 
The formula (\ref{eqn:KRR}) is proven for complex orbifolds in \cite{kawasaki} and for Deligne-Mumford stacks in \cite{toen}. 
 
Suppose $\iy$ is a proper Deligne-Mumford $\mathbb{C}$-stack with a perfect obstruction theory. Let $\mathcal{O}_\iy^{vir}$ be the virtual structure sheaf given by this obstruction theory. For a vector bundle $V$ on $\iy$ the virtual Euler characteristic of $V$, $\chi(\iy, V\otimes \mathcal{O}_\iy^{vir})$, is well-defined. A generalization of (\ref{eqn:KRR}) to virtual Euler characteristics reads
   \begin{align}\label{eqn:vKRR}
    \chi(\iy, V\otimes \mathcal{O}_\iy^{vir}) = \sum_\mu \frac{1}{m_\mu}\int_{[\iy_\mu]^{vir}} Td(T_{\iy_\mu}^{vir})ch \left(\frac{Tr(V)}{Tr(\Lambda^\bullet (N_\mu^{vir})^\vee)}\right).    \end{align}
Here $[\iy_\mu]^{vir}$ is the virtual fundamental class of $\iy_\mu$, $T_{\iy_\mu}^{vir}$ is the virtual tangent bundle of $\iy_\mu$, and $N_{\iy_\mu}^{vir}$ is the virtual normal bundle of $\iy_\mu\subset \iy$. All these are induced naturally from the perfect obstruction theory on $\iy$.

When $\iy$ is a scheme, the formula (\ref{eqn:vKRR}) is proved by \cite{fg} and \cite{cfk}. When $\iy$ is a Deligne-Mumford stack, (\ref{eqn:vKRR}) is proved in \cite{to_k} under somewhat restrictive assumptions. It is expected that these assumptions can be removed via derived algebraic geometry. We were informed that the upcoming work of  Mann-Robalo will address this matter \cite{r}.
 
 \subsection{Kawasaki strata for $\ovkd$}
  We study K-theoretic Gromov-Witten invariants of an orbifold $\ix$ by applying Kawasaki's formula (\ref{eqn:vKRR}) to the moduli spaces $\ovkd$. For this purpose we need to describe symmetries of orbifold stable maps.
  
Symmetries of stable maps to a manifold target space all come from multiple covers i.e. where the map $f$ factors through a degree-$m$ cover $\mathcal{C}\to\mathcal{C}'$. We refer to these symmetries as ``geometric''. 

Given an orbifold stable map $\mathcal{C}\to \mathcal{X}$ with the induced map between coarse moduli spaces denoted by $C\to X$, the automorphism groups $Aut(\mathcal{C}\to \mathcal{X})$ and $Aut(C\to X)$ fit into the following exact sequence:
\begin{equation*}
1\to K\to Aut(\mathcal{C}\to \mathcal{X})\to Aut(C\to X)\to 1.
\end{equation*}
Here $K$ consists of automorphisms of $\mathcal{C}\to \mathcal{X}$ that fix $C\to X$. According to \cite{acv}, these automorphisms, which are called ``ghost automorphisms'',  arise from stacky nodes: there are $\mathbb{Z}_r$ worth of these for each node of the domain curve with $\mathbb{Z}_r$ stabilizer group.  

We refer to \cite{acv} and \cite{av_notes} for more discussions on automorphisms of orbifold stable maps. 
 
     Following \cite{gito} we introduce a dictionary to help us keep track of the Kawasaki strata of $\ovkd$ corresponding to the "geometric symmetries" $g$ and of their contributions to the $J$-function. 
     
  Given a Kawasaki stratum, pick $\mathcal{C}$ a generic domain curve in this Kawasaki stratum and denote the symmetry associated with it by $g$. We call the distinguished first marked point of $\mathcal{C}$ \emph{the horn}. the symmetry $g$ acts with eigenvalue $\zeta$ on the cotangent line at the horn. 
  
 If $\zeta = 1$ the symmetry is trivial on the irreducible component of the curve that carries the horn. We call the maximal connected component of the curve that contains the horn on which the symmetry is trivial \emph{the head}. Notice that the head can be a nodal curve. Heads are parametrized by moduli spaces $\overline{\mathcal{K}}_{0,n'+1}(\ix,d')$ for some $n',d'$. In addition, there might be nodes  connecting  the head with strata of maps with nontrivial symmetries. We call these \emph{the arms}.
  
     Assume now that $\zeta\neq 1$, in which case it is an $m$th root of unity for some $m\geq 2$. Identifying the horn  with $0$, we see that the other fixed point by the $\mathbb{Z}_m$ symmetry can be either a regular point, a marked point or a node. We call the maximal connected component of the curve on which $g^m$ acts trivially  and on which $g$ acts with inverse eigenvalues on the cotangent line at each node \emph{the stem}. The reason why we allow nodes subject to this constraint is because each such node can be smoothed while staying in the same Kawasaki stratum. So stems are chains of (orbifold) $\mathbb{P}^1$'s. In the last $\mathbb{P}^1$ in this chain lies the distinguished point $\infty$, fixed by the symmetry $g$. If it is a node, we call the rest of the curve connected to the stem at that node \emph{the tail}.
      In addition we encounter the following situation: there are $m$-tuples of curves $(\mathcal{C}_1,\ldots , \mathcal{C}_m)$ isomorphic as stable maps, which are permuted by the symmetry $g$. We call these \emph{the legs}. 
      
 We refer to \cite[Figure 1]{gito} for a depiction of these objects.     
      
      Let $B\mathbb{Z}_m$ be the stack quotient $[pt/\mathbb{Z}_m]$. We identify the stem spaces with moduli spaces of maps to the orbifold $\ix \times B\mathbb{Z}_m$. We use the description of maps to $B\mathbb{Z}_m$ given in  \cite{jaki}. 
       Let $\zeta$ be a primitive $m$th root of $1$, and let $$\ix_{0,n+2,d}(\zeta)\subset \overline{\mathcal{K}}_{0,nm+2}(\ix,dm)$$ be the stem space which parametrizes maps $(\mathcal{C},x_0,\ldots ,x_{nm+1},f)$ which factor as $\mathcal{C}\to \mathcal{C'} \to \ix$ where the first map is given in coordinates as $z\mapsto z^m$, $x_0=0\in \mathcal{C}$, $x_{nm+1}=\infty\in \mathcal{C}$ and each $m$-tuple $(x_{mk+1},\ldots ,x_{mk+m})$ is mapped to the same point in $\mathcal{C}'$. Here $\zeta$ is the eigenvalue of the action of the generator $\xi\in\mathbb{Z}_m$ on the cotangent line at $x_0$. 
       
\begin{p1sec3}
       The stem spaces $\ix_{0,n+2,d}(\zeta)\subset \overline{\mathcal{K}}_{0,nm+2}(\ix,dm)$ are identified with (a connected component of) the moduli spaces of orbifold stable maps $\overline{\mathcal{K}}_{0,n+2,(\xi,0,\ldots ,0, \xi^{-1})}(\ix \times B\mathbb{Z}_m,d)$ to the orbifold $\ix\times B\mathbb{Z}_m$. The tuple $(\xi,0,\ldots ,0, \xi^{-1})$ records the sectors of the inertia stack of $B\mathbb{Z}_m$ that are the target of the evaluation maps.   
         \end{p1sec3}
\begin{proof} Stable maps to $B\mathbb{Z}_m$ are described in \cite{jaki} as principal $\mathbb{Z}_m$ bundle on the complement to the set of special points of $\mathcal{C}$, possibly ramified over the nodes in a  \emph{balanced} way, i.e. such that the holonomies around the node of the two branches of the curve are inverse to each other. Representability of the morphism $\mathcal{C}\to \ix$ ensures that the map from the quotient curve $\mathcal{C}'\to \ix$ is a twisted stable map (the order of stabilizers of the points $0,\infty$ is coprime with $m$).
\end{proof}

 \section{Fake Gromov-Witten invariants}\label{sec:fake_GW}
   
 The fake K-theoretic Gromov-Witten invariants of a manifold target space are defined (see \cite{tom} and \cite{gito}) using the terms corresponding to the identity symmetry in Kawasaki's formula. These are also called \emph{fake Euler characteristics}.  For orbifold target spaces we define the fake K-theoretic Gromov-Witten invariants to be the sum of contributions in (\ref{eqn:vKRR}) from ghost automorphisms:
\begin{definition}[fake K-theoretic Gromov-Witten invariants] 
Let $E_1,...,E_n\in K^0(\overline{I\mathcal{X}})$ and let $k_1,...,k_n$ be nonnegative integers. Put
     \begin{align*}
   \langle E_1L^{k_1},..., E_nL^{k_n}\rangle^f_{0,n,d}:=& \int_{[\overline{\mathcal{K}}_{0,n}(\ix,d)]^{vir}}  ch(\otimes^n_{i=1} ev_i^*E_iL_i^{k_i})\cdot Td(\mathcal{T}^{vir})\\
   & + \sum_{r\geq 2}\sum^{r-1}_{k=1}\int_{[\mathcal{Z}_r]^{vir}} ch(\otimes^n_{i=1} ev_i^*E_iL_i^{k_i}) \frac{Td(\mathcal{T}^{vir}_{\mathcal{Z}_r})}{1 - \zeta^k e^{(\ops_+ +\ops_-)/r }}
 \end{align*}  
   where $\mathcal{Z}_r$ parametrizes nodes with non-trivial stabilizer group $\mathbb{Z}_r$, $\mathcal{T}^{vir}$ are the virtual tangent bundles to the moduli space, and we integrate over the corresponding virtual fundamental classes.  
 \end{definition}  
   
 The generating function of fake invariants, called the ``fake potential'', can be defined in the same way as (\ref{eqn:K_potential}). Its domain is the space of formal power series in $(q-1)$ with coefficients in $K^0(\inst)$. 
 
 The purpose of this Section is to describe how the fake K-theoretic Gromov-Witten theory is related to the cohomological Gromov-Witten theory. 
 
 We first explain the loop space formalism for the fake theory. The loop space $\mathcal{K}^f$ is defined as 
  \begin{align*}
  \mathcal{K}^f:= K^0(\inst, \mathbb{C}[[Q]])((q-1)),
  \end{align*}
 with the symplectic form 
 \begin{align*}
 \Omega^f(\mathbf{f}(q), \mathbf{g}(q)):=-Res_{q=1}(\mathbf{f}(q), \mathbf{g}(q^{-1}))\frac{dq}{q}. 
 \end{align*}
 If we write
\begin{align*}
\mathcal{K}_+^f:= K^0(\inst, \mathbb{C}[[Q]])[[q-1]],\qquad \mathcal{K}^f_-=\frac{1}{q-1}K^0(\inst, \mathbb{C}[[Q]])[[\frac{1}{q-1}]],
\end{align*}
then the polarization $$\mathcal{K}^f = \mathcal{K}^f_+ \oplus \mathcal{K}^f_-$$ realizes $\mathcal{K} ^f$ as the cotangent bundle $T^*(\mathcal{K}^f_+)$.
      
To identify $\mathcal{K}^f$ and $\mathcal{H}$ in the case of  manifolds targets $\ix$ one uses the Chern character which maps isomorphically $K^0(\ix)$ to $H^*(\ix)$. In our situation there is an isomorphism 
$ch \circ Tr:K^0(\overline{\mathcal{X}}_\mu)\to H^*(I\overline{\ix}_\mu)$, whose restriction to the identity component $\overline{\ix}_\mu$ is the usual Chern character. We can pick a basis $\{\Phi_a\}$ of $K^0(\overline{\ix}_\mu)$ such that $(ch\circ Tr)(\Phi_a)$ is supported on only one connected component of $I\overline{\ix}_\mu$. However notice that due to the presence of $ch (\Phi_a)$  in the correlators
only the classes supported on  the identity component give nontrivial contributions. With this choice of basis in mind we will imprecisely say that we use  the Chern character to identify $K^0(\overline{\ix}_\mu)$ with $H^*({\ix}_\mu)$.

  We extend this identification to the spaces $\mathcal{K}^f$ and $\mathcal{H}$ by: 
       \begin{align*}
       ch: \mathcal{K}^f \to \mathcal{H}, \quad q\mapsto e^{z}  .
        \end{align*}
        
We use the formalism of twisted Gromov-Witten theory (\cite{coatesgiv}, \cite{tseng}, \cite{to1}) to see how the cone of the fake theory $\mathcal{L}^f$ is related to the cohomological one. First we need to  describe   $\mathcal{T}^{vir} \in K^0(\ovkd)$:
         \begin{align}
          \mathcal{T}^{vir} = \pi_*(ev^*_{n+1} T_\ix)  - \pi_*(\Omega^\vee(-D)) .
         \end{align}
where $D$ is the divisor of marked points. Basically the first summand accounts for deformations of the map while the second has fiber over a point $(\mathcal{C}, \Sigma_i , f)$ being $$H^1(C,\Omega^\vee (-D))- H^0(C,\Omega^\vee (-D)),$$ the first term accounts for deformations of complex structure (they are unobstructed), from which we substract the infinitesimal automorphisms which fix the marked points. The deformation theory of orbicurves is described in \cite{abgrvi}.
 
 We can imitate the computation in \cite[Section 5]{to1} to conclude that
  \begin{align} \label{tangentv}
   \mathcal{T}^{vir} = \pi_*(ev^*_{n+1} T_\ix-1)   -\pi_*(L^{-1}_{n+1}-1) -(\pi_*i_*\mathcal{O}_\mathcal{Z})^\vee.
     \end{align}
  Here $\mathcal{Z}$ is the codimension $2$ nodal locus. We denote by $L_+$ and $L_-$ the cotangent lines at points on the coarse curves after separating the nodes and their first Chern classes are denoted by $\ops_+$ and $\ops_-$. 
  
  According to \cite{to1}, twisting by characteristic classes of each of the three summands contribute differently to the formalism. The index bundle twisting rotates the cone by a loop group transformation which we denote by $\triangle$.  We apply the twisting theorem of \cite{tseng} to the twisting data $(Td, T_\ix)$  to write down the formula for 
 $\triangle$. Let $s(x)$ be given by
   \begin{align*}
   \frac{x}{1-e^{-x}} = e^{s(x)}= e^{\sum s_k x^k/k!}.
\end{align*}
    On the connected component $\overline{\ix}_\mu$ of $\inst$ denote by $T^{(l)}_\ix$ the eigenbundle\footnote{For simplicity, the dependence on $\mu$ is omitted in the notation for these eigenbundles.} on which the group element acts with eigenvalue $e^{2\pi i l/r}$. Let $x^{(l)}_i$ be the Chern roots of $T^{(l)}_\ix$ . Then if we denote the restriction of $\triangle$ to  $H^*(\ix_\mu)((z))$ by $\triangle_\mu$, we have 
    \begin{align*}
    \triangle_\mu = \prod_l \prod_{x^{(l)}_i} \prod^\infty_{b=1}\frac{x^{(l)}_i-bz+zl/r_\mu }{1-e^{-x^{(l)}_i+bz-zl/r_\mu}}
    \end{align*}
  where the logarithm of the RHS is a formal series obtained by the procedure:   
\begin{align} \label{eumc}
\sum_{b>0}s\left(x+\frac{l}{r}z -bz\right)=\frac{e^{zl/r\partial_x}z\partial_x}{e^{z\partial_x}-1}(z\partial_x)^{-1} s(x)=\sum_{b\geq 0}\frac{B_b(l/r)}{b!}(z\partial_x)^{b-1}s(x).     
\end{align} 
  Also recall that the Bernoulli polynomials  $B_b(x)$ are defined by
  \begin{align*}
  \frac{te^{tx}}{e^t-1}=\sum_{b\geq 0}B_b(x)\frac{t^b}{b!}.
  \end{align*}
  
  The second summand in the formula (\ref{tangentv}) accounts for a dilaton shift change from $-z$ to $1-e^z=1-q$, see \cite{to1}. 
  
  The third summand (nodal twisting) accounts for a change of polarization, and so do the contributions coming from ghost automorphisms. For each  nodal locus $\mathcal{Z}_r$ we have contributions:
  \begin{align*}
 \sum^{r-1}_{k=1}\int_{\mathcal{Z}_r} ch(... ) \frac{Td(T_{\mathcal{Z}_r})}{1 - \zeta^k e^{(\ops_+ +\ops_-)/r }} .
  \end{align*}
(The virtual normal bundle to a nodal locus $\mathcal{Z}_r$ has first Chern class $(\ops_+ +\ops_-)/r$.)  
  
Notice that 
  \begin{align}
  \sum^{r-1}_{k=1}\frac{1}{1-\zeta^k x} = \frac{r}{1-x^r} - \frac{1}{1-x}.\label{403}
  \end{align}
 Pushing forward the correlators on $\ovkd$ we get integrals of the form:
 \begin{align}
  \int_{\ovkd} ch(..)Td(\mathcal{T}_{\mathcal{Z}_r})\pi_*i_*\left[\frac{r}{1-e^{\ops_+ +\ops_-}} - \frac{1}{1-e^{(\ops_++\ops_-)/r}}\right]. 
  \end{align} 
  To describe this in the formalism of the twisted Gromov-Witten invariants we need to find a multiplicative class $\mathcal{C}$ satisfying
  \begin{align}
  \mathcal{C}(-(\pi_*i_*\mathcal{O}_{\mathcal{Z}_r})^\vee) = \pi_*i_*\left[\frac{r}{1-e^{\ops_+ +\ops_-}} - \frac{1}{1-e^{(\ops_++\ops_-)/r}} \right].
  \end{align}
  The proof in \cite[Section 2.5.4]{tom} carries over to show that:
  \begin{align*}
  ch_k(-(\pi_*i_*\mathcal{O}_{\mathcal{Z}_r})^\vee )=\pi_*i_*\left[\frac{(-\ops_+ -\ops_-)^{k-1}}{r^{k-1}k!}\right]. 
  \end{align*}
 To find $\mathcal{C}$ we pull back by the map $\pi \circ i$, which multiplies the classes above by the Euler class of the normal bundle to $\mathcal{Z}_r$. We get:
  \begin{align}
  \mathcal{C}\left(\frac{-\ops_+ -\ops_-}{r}\right)= \frac{(-\ops_+ - \ops_-)}{1-e^{\ops_+ +\ops_-}} + \frac{\ops_+ +\ops_-}{r(1-e^{(\ops_++\ops_-)/r})}.
  \end{align}
  The characteristic class $\mathcal{C}$ is hence defined by (let $L_z$ be the line bundle with $c_1(L_z)=z$ )
  \begin{align}
  \mathcal{C}(- L_z) = \frac{-rz}{1-e^{rz}} +\frac{z}{1-e^{z}}. \label{405}
  \end{align}
   This needs to be added to the nodal contribution coming from the tangent bundle $Td^\vee(-\pi_*i_*\mathcal{O}_{\mathcal{Z}_r})$:
  \begin{align}
  \mathcal{C}_0 (- L_z) = Td^\vee (L_z) = \frac{z}{e^z-1}.
  \end{align} 
 This cancels the second term in (\ref{405}), so we are left with the first, which is $Td^\vee(L^{\otimes r}_z)$. Hence the nodal twisting is given by  $$Td^\vee (-\pi_*i_*\mathcal{O}_{\mathcal{Z}_r} )^{\otimes r}.$$ According to \cite[Example 4.2]{to1} a basis for the negative space $\mathcal{K}_-^f$ of the polarization of the fake theory is given by $\{\Phi_a\frac{q^{k}}{(1-q)^{k+1}}\}$ where $\Phi_a$ runs a basis of $\oplus_\mu K^0(\overline{\ix}_\mu)$ and $k$ runs from $0$ to $\infty$.
 
Putting the above arguments together, it shows that if we define the fake $J$ function as
\begin{align*}
\mathcal{J}^f (\mathbf{t}(q)) =1-q+\mathbf{t}(q)+\sum_{n,d,a}\frac{Q^d}{n!}\Phi_a\la \frac{\Phi^a}{1-qL}, \mathbf{t}(L), \ldots ,\mathbf{t}(L)\ra^f_{0,n+1,d}
\end{align*}
then it is the differential of the genus $0$ fake potential  and if we denote its range by $$\mathcal{L}^f\subset \mathcal{K}^f,$$ then we have 
\begin{th1sec4}\label{th1sec4l}
       The cones $\mathcal{L}^f$ and $\mathcal{L}^H$ are related by: 
       \begin{align*}
       ch(\mathcal{L}^f) = \triangle\mathcal{L}^H.
       \end{align*}
       \end{th1sec4}
       
\section{The main theorem}\label{sec:main_thm}
\subsection{Statements}
For elements $\mathbf{f}\in \mathcal{K}$, we denote by $\mathbf{f}_\zeta$ expansions of $\mathbf{f}$ into Laurent series in $(1-q\zeta)$. We regard $\mathbf{f}_\zeta$ as elements in  $\mathcal{K}^\zeta$ - the space of formal Laurent power series in $(1-q\zeta)$. The map $q\mapsto q\zeta^{-1}$ defines an isomorphism between $\mathcal{K}^\zeta$ and $\mathcal{K}^f$. Via this isomorphism $\mathcal{K}^\zeta$ becomes a symplectic space with a distinguished polarization.
 
The $J$-function $\mathcal{J}_\ix$, as defined in (\ref{eqn:J_K}) is a rational function in $q$ with poles at all roots of unity. Recall that the range of $\mathcal{J}_\ix$ is denoted by $\mathcal{L}$, which is an overruled Lagrangian cone, see Theorem \ref{thm:cone_K}. The following is the main result of this paper, which characterizes points on the cone $\mathcal{L}$ in terms of the cone $\mathcal{L}^f$ of the fake theory.
    
\begin{theorem}\label{main1}
Let $\mathcal{L}\subset \mathcal{K}$ be the overruled Lagrangian cone of quantum K theory on $\ix$. Let $\mathbf{f}\in \mathcal{K}$. Then $\mathbf{f}\in \mathcal{L}$ if and only if the following hold:
\begin{enumerate}
\item  $\mathbf{f}_\zeta$ does not have poles unless $\zeta\neq 0,\infty$ is a root of unity.

\item $\mathbf{f}_1$ lies on $\mathcal{L}^{f}$.
           
\item 
For every $\mu\in \mathcal{I}$ introduce the series obtained by applying the procedure described in (\ref{eumc}) (as $q = e^z\rightarrow 1$) of the infinite products: 
             \begin{align*}
             &\Box_{0,\mu}:= \prod_{i,l} \prod^\infty_{b=1}\frac{x^{(l)}_i+lz/r_\mu  -bz}{1-e^{-mx^{(l)}_i +bmz-zl/r_\mu}}\\
             &\Box_{1,\mu}:=\prod_{i,l} \prod^{m-1}_{k=0}\prod^\infty_{b=1}\frac{x^{(l)}_i+lz/r_\mu -bz}{1 -\zeta^k e^{-x^{(l)}_i- l_\mu z +bz}}
             \end{align*}
Let $\Box_0:=\oplus_{\mu\in \mathcal{I}}\Box_{0,\mu}$, and $\Box_1:=\oplus_{\mu\in \mathcal{I}}\Box_{1,\mu}$. 
Here $P_i^{(l)}=e^{-x_i^{(l)}}$ acts by multiplication in $K^0(\overline{\ix}_\mu)$ .
 Let $\mathcal{T}_m$ be the linear space of Definition \ref{deft} below. If $\zeta\neq 1$ is a primitive $m$ root of $1$, then $$ \mathbf{f}_\zeta(q^{1/m}\zeta^{-1})\in \Box_1\Box_0^{-1}\mathcal{T}_m$$ .     
\end{enumerate}
\end{theorem}
 \begin{newrem}
 {\em The numbers $l_\mu$ in the formulae or $\Box_{1,\mu}$ are  explicitly defined as $$ l_\mu : = \langle \frac{l}{r_\mu} - \frac{k}{m}\rangle,$$  
 where $ \langle n \rangle$ denotes the fractionary part of $n$.  }
 
 \end{newrem}
\begin{deftang}\label{deft}
{\em   Let $T^{f}_0$ be the tangent space to $\mathcal{L}^{f}$ at $\mathcal{J}_1(0)$ (which according to Theorem \ref{main1} (2). lies on $\mathcal{L}^{f}$), considered as the image of a map $T(q,Q):\mathcal{K}^f_+\to \mathcal{K}^f$. Let $\psi^m$ be the automorphism of $H^*(\overline{\ix}_\mu)$ by $a\mapsto m^{deg(a)/2}a$ and on $q\mapsto q^m$. Denote by $\psi^{\frac{1}{m}}$  the isomorphism of $\mathcal{K}^f_+$ which is the inverse of $\psi^m$ on $K^0(\overline{\ix}_\mu)$ (under the identification of $K^0(\overline{\ix}_\mu)$  and $H^*(\overline{\ix}_\mu)$) and acts as $q\mapsto q^{1/m}$ (but not on $Q^d$). Then define }
               \begin{align*}
               \mathcal{T}_m :=\text{ Image of } \psi^m \circ  T(q, Q^m) \circ \psi^{1/m}: \mathcal{K}^f_+ \to \mathcal{K}^f,
               \end{align*} 
 
\end{deftang}

\begin{correct}
{\em The formulation of the analogous theorem in \cite{gito} is slightly incorrect. Our theorem specializes to the correct statement in the case when $\ix$ is a manifold}.
\end{correct}

\subsection{Some applications} \label{sec:app}
We discuss some applications of Theorem \ref{main1}.
  
    \begin{c2sec7}
  The K-theoretic Gromov-Witten invariants can be computed from the cohomological ones.
     \end{c2sec7}         
  \begin{proof} 
  Proposition \ref{p8sec5l} below determines the $J$-function from the fake $J$-function. Theorem \ref{th1sec4l} in principle expresses the fake invariants in terms of cohomological ones. The Corollary follows by combining these two results.
\end{proof}  

The next application concerns the $\mathcal{D}$-module structure in quantum K-theory. Pick a basis $p_a\in H^2(\ix)$ and let $P_a$ be line bundles such that $e^{-p_a}=P_a$. Let $\mathcal{D}_q$ denote the algebra of differential operators generated by powers of $P_a$ and $Q^d$, for $d$ lying in the Mori cone of $\ix$. Define a representation of $\mathcal{D}_q$ on $\mathcal{K}$ using the operators $P_a q^{Q_a\partial_{Q_a}}$ and multiplication by $Q^d$ in the Novikov ring. Then
\begin{theorem}
 The tangent spaces to the cone $\mathcal{L}$ are preserved by this action.   
    \end{theorem}
 \begin{proof}    
 The proof is the same as the corresponding one in \cite[Section 9]{gito}.
 \end{proof}

   \section{Proof of the main Theorem \ref{main1}}\label{sec:proof_main}
  
 \subsection{Expansion at $1$}  
The first part of Theorem \ref{main1} is obvious from Kawasaki's formula. For the second part we prove the following 
 \begin{p1sec5}\label{p1sec5l}
The localization of the $J$-function at $q=1$ lies on the cone $\mathcal{L}^f$.
 \end{p1sec5}
\begin{proof} Denote by $\widetilde{\mathbf{t}}$ the sum of correlators for which $\zeta\neq 1$. Then we have:
     \begin{align}
   \mathcal{J}_1 = 1-q+ \mathbf{t} + \widetilde{\mathbf{t}} +\sum_{n,m,d}\frac{Q^d}{n!}\frac{1}{m!}\sum_{a}\phi_a\la \frac{\phi^a}{1-qL},\mathbf{t}(L),\ldots \mathbf{t}(L),\widetilde{\mathbf{t}}(L),\ldots ,\widetilde{\mathbf{t}}(L)\ra^f_{0,n+m+1,d}  . \label{j11}
     \end{align}
where there are $n$ occurrences of $\mathbf{t}$ and $m$ of $\widetilde{\mathbf{t}}$ in the correlators. 
    
   The reason why this is true is because each of the special points on the irreducible component that carries the horn is either a marked point or a node connecting it to an arm. If it is a marked point the input in the correlator is $\mathbf{t}(L)$. If it is a node, assume it has stabilizer group of order $r$: it is known that the Euler class of the normal direction to the stratum which smooths the node is $1-(L_+L_-)^{1/r}$ where $L^{1/r}_-, L^{1/r}_+$  are the cotangent lines\footnote{Strictly speaking $L_\pm^{1/r}$ do not exist on $\overline{\mathcal{K}}_{0,n}(\ix,d)$. However, since fake theory involves the Chern characters $ch(L_\pm^{1/r})$, not the actual line bundles, we may pretend that these line bundles exist.} to the head and arm respectively. Therefore the input is:
    \begin{align*}
    \sum_a \frac{\phi^a\otimes \phi_a}{1-L^{1/r}_-Tr(L^{1/r}_+)}.
    \end{align*} 
(the trace is taken with respect to the $\mathbb{Z}_m$ symmetry acting on the outgoing stem independently). In addition we have $r-1$ contributions from ghost automorphisms corresponding to that node, they have input 
     \begin{align*}
        \sum_a \frac{\phi^a\otimes \phi_a}{1-\zeta^k L^{1/r}_-Tr(L^{1/r}_+)}, \quad 1\leq k\leq r-1,
        \end{align*}
    for $\zeta$ a primitive $r$ root of unity. They add up to (see (\ref{403}))
    \begin{align*}
       \sum_a \frac{r\phi^a\otimes \phi_a}{1-L_-Tr(L_+)}.
       \end{align*} 
     
The node  becomes the horn for the integral on the arm. When we sum after all such possibilities, the contribution is $\widetilde{\mathbf{t}}(q)$ at the point $q=L_-$. The factor $\frac{1}{n!m!}$ in front of the correlators is combinatorial, it accounts for choosing which are the marked points and occurs because $\frac{1}{(n+m)!}\binom{n+m}{m}=\frac{1}{n!m!}$.  We can rewrite $(\ref{j11})$ as: 
\begin{align}
\mathcal{J}_1 = 1-q+ \mathbf{t} + \widetilde{\mathbf{t}} +\sum_{a,n',d}\phi_a\frac{Q^d}{n'!}\la \frac{\phi^a}{1-qL},\mathbf{t}(L)+\widetilde{\mathbf{t}}(L),\ldots , \mathbf{t}(L)+\widetilde{\mathbf{t}}(L)\ra^f_{0,n'+1,d} = \mathcal{J}^f(\mathbf{t}+\widetilde{\mathbf{t}}). \label{j12}
\end{align}
This proves the proposition.
 \end{proof}

\subsection{Expansions at other roots of $1$} 
\begin{p3sec5} \label{p3sec5l}
Let $\zeta\neq 1$ be a root of unity. The localization $\mathcal{J}_\zeta(q^{1/m}\zeta^{-1})$ is a tangent vector to the cone of certain ``twisted'' fake theory, after identifying the loop spaces using the Chern isomorphism. The application point is the leg $\mathbf{T}$. 
     \end{p3sec5} 
 \begin{proof} Denote by $\delta \mathbf{t}(q)$ the sum of terms in $\mathcal{J}_\ix$ which  do not have a pole at $q=\zeta^{-1}$. Then we can write:
    \begin{align}
    \mathcal{J}_\zeta(\mathbf{t})=\delta\mathbf{t}(q) + \sum_{a,n,d}\phi_a\frac{Q^{dm}}{n!}\la \frac{\phi^a}{1-q\zeta L^{1/m}},\mathbf{T}(L),\ldots ,\mathbf{T}(L) ,\delta\mathbf{t}(\zeta^{-1}L^{\frac{1}{m}});Tr(\Lambda^* N_{0,n,d})\ra^{\ix/\mathbb{Z}_m,f}_{0,n+2,d}
    \end{align} 
   where $N_{0,n,d}$ is the normal bundle to each stem space. Remember that $g$ acts by $\zeta$ on the cotangent line at the first marked point, which explains the denominator $1-q\zeta L^{1/m}$ of the input at that point in the correlators. We now explain the input  $\delta\mathbf{t}(\zeta^{-1}L^{\frac{1}{m}})$  at the second branch point $\infty$. If $\infty$ is a marked point, then the input is $\mathbf{t}(\zeta^{-1}L^{1/m})$. If it is a nonspecial point of the original curve, than we claim the input is $1-\zeta^{-1}L^{1/m}$. For this look at the diagram (assume $n=0$ for simplicity, since the presence of legs does not change the following argument):  
      \begin{displaymath}
                 \begin{CD}
              \ix_{0,2,d}(\zeta)    @> i >>     \overline{\mathcal{K}}_{0,2}(\ix,dm)    \\
             @V ft_2 VV                @V ft_2 VV  \\
              \ix_{0,2,d}(\zeta)    @> i >>     \overline{\mathcal{K}}_{0,1}(\ix,dm)  .\\
              \end{CD}
         \end{displaymath}
      The restriction of $ft_2$ to the Kawasaki stratum $\ix_{0,2,d}(\zeta)$ is an isomorphism so the conormal bundle $\overline{N}^\vee$ of $\ix_{0,2,d}(\zeta)$ in $\overline{\mathcal{K}}_{0,2}(\ix,dm)$ is the direct sum of the conormal bundle of $\ix_{0,2,d}(\zeta)$ in $D_1:=\sigma_1(\overline{\mathcal{K}}_{0,1}(\ix,dm))$ and the conormal bundle of $D_1$ in $\overline{\mathcal{K}}_{0,2}(\ix,dm)$. Here $D_1\subset \overline{\mathcal{K}}_{0,2}(\ix,dm)$ is the image of the section of the first marked point.  The conormal bundle of $\ix_{0,2,d}(\zeta)$ in $D_1$ is the same as the conormal bundle of $\ix_{0,2,d}(\zeta)$ in $\overline{\mathcal{K}}_{0,1}(\ix,dm)$. We denote it by $N^\vee$. Taking equivariant Euler classes gives:
       \begin{align*}
          \Lambda^*(\overline{N}^\vee)= \Lambda^*(N^\vee)(1-\zeta^{-1}L^{1/m}_2).
        \end{align*}
      Hence integrals on $\ix_{0,2,d}(\zeta)$ viewed as a Kawasaki stratum in $\overline{\mathcal{K}}_{0,1}(\ix,dm)$ can be expressed as integrals on the stem space with the input $1-\eta^{-1}L^{1/m}$ at $\infty$. Finally  when $\infty$ is a node, then the input is the polar part of $\delta\mathbf{t}(\zeta^{-1}L^{1/m})$.
      
The reason why we view  $\mathcal{J}_\zeta(q^{1/m}\zeta^{-1})$ as a tangent vector to a Lagrangian cone is that we can identify tangent spaces to cones of theories with first order derivatives of their $J$-functions. Taking the derivative of $J$ in the direction of $\vec{v}(q)$ replaces the input by $\vec{v}(q)$ and one seat in the correlators by $\vec{v}(L)$.
        
Although the correlators are on $\ix/\mathbb{Z}_m$, it is explained below that we can identify this generating series with a tangent space to the cone of a twisted theory on $\ix$. 
\end{proof}
      We now analyze the leg contribution:
   \begin{l1sec5}  \label{l1sec5l}
 Let $\widetilde{\mathbf{T}}$ be the arm contributions computed at the input $\mathbf{t}(q)=0$. Then: 
    \begin{align*}
   \mathbf{T}(L) = \psi^m \left( \widetilde{\mathbf{T}}(L)\right) .
    \end{align*}
    \end{l1sec5}
\begin{proof} 
The symmetry $g^m$ acts nontrivially at the cotangent line to each copy of the leg because otherwise they are degenerations inside a higher degree stem space.
  When we sum over all possible contributions for each copy of the leg we get the arm contributions. The legs are not allowed to contain marked points, hence the input is $\mathbf{t} = 0$. Since we have $m$ copies of each leg, the contribution is $Tr(g\vert \widetilde{\mathbf{T}}(L)^{\otimes m})$. The proposition then follows from the next Lemma.
  \end{proof}    
    \begin{l2sec5}[see \cite{gito}, Section 7]
    \emph{Let $V$ be a vector bundle. Then} 
    \begin{align*}
   Tr\left(g\vert V^{\otimes m}\right) =\psi^m (V) .   
    \end{align*}   
     \end{l2sec5} 
 
We thus proved the first two conditions in Theorem \ref{main1}.
 
 \subsection{Stem theory}    
 To prove the last condition in Theorem \ref{main1}, we need to describe the tangent and normal bundles to the stem spaces 
 $$\overline{\mathcal{M}}:=\ix_{0,n+2,d}(\zeta)$$ 
 in $ \overline{\mathcal{K}}_{0,mn+2}(\ix,md)$ in terms of the universal family $\pi : \mathcal{U}\to \overline{\mathcal{M}}$. Denote by 
 $$\widetilde{\pi}:\overline{\mathcal{K}}'_{0,mn+3}(\ix,md)\to \overline{\mathcal{K}}_{0,mn+2}(\ix,md)$$ 
 the universal family over $\overline{\mathcal{K}}_{0,mn+2}(\ix,md)$. Denote by $\widetilde{\mathcal{U}}:=\widetilde{\pi}^{-1}(\overline{\mathcal{M}})$. Then the map 
 $$\widetilde{\pi}: \widetilde{\mathcal{U}}\to \overline{\mathcal{M}}$$ 
 is a $\mathbb{Z}_m$-equivariant lift of $\pi$, i.e. each fiber of $\widetilde{\pi}$ is a ramified $\mathbb{Z}_m$ cover of the corresponding fiber of $\pi$. There are also evaluation maps at the last marked point (we omit the index) $ev: \mathcal{U}\to \ix/\mathbb{Z}_m$ and its $\mathbb{Z}_m$ lift $\widetilde{ev}:\widetilde{\mathcal{U}}\to \ix $.
    
Recall that the virtual tangent bundle is given by formula (\ref{tangentv})
     \begin{align*}
    \mathcal{T}:=\widetilde{\pi}_*(ev^*T_\ix -1) -\widetilde{\pi}_*(L^{-1}_{nm+3}-1) -(\widetilde{\pi}_*i_*\mathcal{O}_\mathcal{Z})^\vee.
     \end{align*}
We need to compute the trace of $\xi\in\mathbb{Z}_m$ on each piece of this bundle. Denote by $\mathbb{C}_{\zeta^k}$ the $\mathbb{Z}_m$ representation $\mathbb{C}$ where $\xi$ acts by multiplication by $\zeta^k$. K-theoretic push-forwards on orbifolds considered as global quotients extract invariants, so the piece of  $\widetilde{\pi}_*(\widetilde{ev}^*T_\ix )$ on which $T\ix$ acts by $\zeta^{-k}$ can be expressed as $\pi_*ev^*(T_\ix\otimes\mathbb{C}_{\zeta^k})$.  Therefore the trace is given by: 
        \begin{align*}
        Tr\left(\widetilde{\pi}_*(\widetilde{ev}^*T_\ix )\right)=\sum^{m-1}_{k=0}\zeta^{-k} \pi_*ev^*(T_\ix\otimes \mathbb{C}_{\zeta^k}) .
        \end{align*}
   Of course the term $k=0$ corresponds to the tangent bundle and the others to the normal bundle. Similarly:
      \begin{align*}
     Tr\left( \widetilde{\pi}_*(L^{-1}_{mn+3})\right) =  \sum^{m-1}_{k=0} \zeta^{-k} \pi_*( L_{n+3}^{-1}ev^*\mathbb{C}_{\zeta^k}).
      \end{align*}
      
Let $\widetilde{\mathcal{Z}}\subset \widetilde{\mathcal{U}}$ and $\mathcal{Z}\subset\mathcal{U}$ be the nodal loci. We distinguish two types of nodes. When the node is a balanced ramification point of order $m$, the tangent bundle is one dimensional and is invariant (its $K$ theoretic Euler class class is $1-L_+^{1/m}L_-^{1/m}$). If we denote by $\mathcal{Z}_\xi$ this nodal locus, downstairs this corresponds to twisting by the class $Td(-\pi_*i_{\xi*}\mathcal{O}_{\mathcal{Z}_\xi})^\vee$.  
      
If the node is unramified, then the covering curve has a $\mathbb{Z}_m$ symmetric $m$-tuple of nodes. The smoothing bundle has dimension $m$; it contains a one dimensional subspace which is tangent to the stratum and a $m-1$ dimensional subspace normal to it. Denote by $\mathcal{Z}_0 , \widetilde{\mathcal{Z}}_0$ the corresponding nodal loci. We claim that:
\begin{align}\label{eqn:inv_part}
\left( (\widetilde{\pi}_*i_{0*}\mathcal{O}_{\widetilde{\mathcal{Z}}_0})\otimes \mathbb{C}_{\zeta^{-1}}\right)^{\mathbb{Z}_m} = \pi_*\left(\varphi^*\mathbb{C}_{\zeta^{-1}}\otimes i_{0*}\mathcal{O}_{\mathcal{Z}_0} \right).
\end{align}
\begin{proof}[Proof of (\ref{eqn:inv_part})] 
We think of the sheaf $i_{0*}\mathcal{O}_{\widetilde{\mathcal{Z}_0}}$ as the trivial bundle on $\widetilde{\mathcal{Z}}_0$. The map $p: \widetilde{\mathcal{Z}_0}\to \mathcal{Z}_0$ is an $m$ cover. The pushforward of a vector bundle $E$ along this map is the vector bundle  $E\otimes \mathbb{C}^m$, where the transition matrices map $v\otimes e_i\mapsto v\otimes e_{i+1}$, or equivalently it is the regular $\mathbb{Z}_m$ representation acting on the direct sum of $m$ copies of $E$. For each $\zeta $ the subbundle on which the generator of $\mathbb{Z}_m$ acts with eigenvalue $\zeta$ is isomorphic to $E$. Applying this to the trivial bundle proves the claim because:   
   \begin{align*}
   \left( (\widetilde{\pi}_*i_{0*}\mathcal{O}_{\widetilde{\mathcal{Z}}_0})\otimes \mathbb{C}_{\zeta^{-1}}\right)^{\mathbb{Z}_m} &= \left(\widetilde{\pi}_*(i_{0*}\mathcal{O}_{\widetilde{\mathcal{Z}}_0}\otimes \mathbb{C}_{\zeta^{-1}})\right)^{\mathbb{Z}_m}\\
   &=  \pi_*\left( p_* (i_{0*}\mathcal{O}_{\widetilde{\mathcal{Z}}_0}\otimes \widetilde{\varphi}^*\mathbb{C}_{\zeta^{-1}})\right)^{\mathbb{Z}_m} \\
   &= \pi_* \left(\varphi^*\mathbb{C}_{\zeta^{-1}}\otimes i_{0*}\mathcal{O}_{\mathcal{Z}_0} \right).
   \end{align*}
           \end{proof}
           
To prove the last part of Theorem \ref{main1} we need to introduce more notation and generating series: we index the components of $\inst\times \overline{IB\mathbb{Z}}_m$ by pairs $(g_\mu, h)$, where $g_\mu$ indexes the connected component of $\inst$ and $h \in \mathbb{Z}_m$. We write:
                \begin{align*}
                H^*(\inst\times\overline{IB\mathbb{Z}}_m ,\mathbb{C}): = \oplus_{h\in \mathbb{Z}_m }H^*(\inst/\mathbb{Z}_m ,\mathbb{C}) e_h .
                \end{align*}
For cohomology classes in the identity (of $\mathbb{Z}_m$) sector we often omit the element $e_1$ from the notation. Fix $g\in \mathbb{Z}_m$ the element which indexes the connected component of $\overline{IB\mathbb{Z}}_m$ where the evaluation map at the first marked point of the stem spaces lands.

 We now introduce the following generating series for Gromov-Witten theory of $\ix\times B\mathbb{Z}_m$:
              \begin{align*}
              J_{\ix/\mathbb{Z}_m}&:= -z + \mathbf{t}(z) + \sum_a \widetilde\phi_a \sum_{n,d}\frac{Q^d}{n!} \la \frac{\widetilde{\phi^a}}{-z-\ops},\mathbf{t}(\ops),\ldots , \mathbf{t}(\ops)  \ra^{\ix/\mathbb{Z}_m}_{0,n+1,d}\\
               \delta J_{\ix/\mathbb{Z}_m}&:= \delta\mathbf{t}(z) + \sum_a \widetilde\phi_a e_{g^{-1}} \sum_{n,d}\frac{Q^d}{n!}\la \frac{\widetilde{\phi^a}e_g}{-z-\ops},\mathbf{t}(\ops),\ldots , \delta \mathbf{t}(\ops)e_{g^{-1}}  \ra^{\ix/\mathbb{Z}_m}_{0,n+2,d}
              \end{align*}
where $\{\widetilde{\phi_a} \}$ and $\{\widetilde{\phi^a}\}$ are dual basis with respect to the Poincar\'e orbifold pairing on $\overline{I(\ix/\mathbb{Z}}_m)$. It follows from \cite{jaki} that 
\begin{align*}
J_{\ix/\mathbb{Z}_m}&=J^H_\ix, \qquad \text{where} \qquad J_\ix(z,\mathbf{t}(z))\in H^*((1,\inst))((z^{-1}))\simeq \mathcal{H}_\ix ,\\
\delta J_{\ix/\mathbb{Z}_m}&= \delta\mathbf{t}(z) + \sum_a \widetilde\phi_a \sum_{n,d}\frac{Q^d}{n!}\la \frac{\widetilde{\phi^a}}{-z-\ops},\mathbf{t}(\ops),\ldots , \delta \mathbf{t}(\ops)  \ra^{\ix}_{0,n+2,d}.
\end{align*}
We now define their twisted counterparts :
\begin{align*}
J^{tw}_{\ix/\mathbb{Z}_m}&:= -z + \mathbf{t}(z) + \sum_a \widetilde\phi_a \sum_{n,d}\frac{Q^d}{n!} \la \frac{\widetilde{\phi^a}}{-z-\ops},\mathbf{t}(\ops),\ldots , \mathbf{t}(\ops);\Theta_{0,n,d}  \ra^{\ix/\mathbb{Z}_m}_{0,n+1,d} ,\\
\delta J^{tw}_{\ix/\mathbb{Z}_m}&:= \delta \mathbf{t}(z) + \sum_a \widetilde\phi_a \sum_{n,d}\frac{Q^d}{n!}\la \frac{\widetilde{\phi^a}e_g}{-z-\ops},\mathbf{t}(\ops),\ldots , \delta \mathbf{t}(\ops)e_{g^{-1}} ;\Theta_{0,n,d} \ra^{\ix/\mathbb{Z}_m}_{0,n+2,d} .
\end{align*}
Here $\Theta_{0,n,d}$ is the twisting data coming from the index bundle occurring in the stem theory:
             \begin{align*}
            \Theta_{0,n,d}:= & Td(\pi_*ev^*(T_\ix)) \prod_{k=1}^{m-1} td_{\zeta^k}(\pi_*ev^*(T_\ix\otimes \mathbb{C}_{\zeta^k})) .
               \end{align*} 
 and its inclusion in the correlators means we multiply it with the integrand. The classes $td_\lambda$ is the unique multiplicative characteristic class given on line bundles by   
     \begin{align*}
      td_\lambda (L) = \frac{1}{1-\lambda e^{-c_1(L)}}.
     \end{align*}

\begin{p5sec5}\label{prop5s5}
Let $\Box_0, \Box_1$ be the operators in Theorem \ref{main1}. The series $J_{\ix/\mathbb{Z}_m}^{tw}$ lies in the overruled Lagrangian cone $\Box_0\mathcal{L}_\ix^H$. The series $\delta J_{\ix/\mathbb{Z}_m}^{tw}$ lies in the tangent space $\Box_1 \mathcal{T}_{\Box_0^{-1}J_{\ix/\mathbb{Z}_m}^{tw}}\mathcal{L}_\ix^H$.
\end{p5sec5} 
\begin{proof} 
This follows from the twisting theorem of \cite{tseng}; notice that the cone for the cohomological Gromov-Witten theory of $\ix\times B\mathbb{Z}_m$ is a product of $m$ copies of the cone for the cohomological Gromov-Witten theory of $\ix$; similarly  tangent spaces to $\mathcal{L}^H_{\ix/\mathbb{Z}_m}$ are direct sums of $m$ copies of tangent spaces to $\mathcal{L}^H_\ix$. Notice that the nuumbers $l_\mu$ in the operator $\Box_1$ occur due to the action 
of pairs $(g_\mu,g)$ on the bundle $T_\mathcal{X}\otimes \mathcal{C}_{\zeta^k}$.

The range of $J_{\ix/\mathbb{Z}_m}^{tw}$ is the part of the untwisted sector (with respect to elements in $\mathbb{Z}_m$) of the cone $\mathcal{L}^{tw}$, which gets rotated  by $\Box_0$. The tangent vector $\delta J_{\ix/\mathbb{Z}_m}^{tw}$ on the other hand belongs to the sector indexed by $g^{-1}$. 
\end{proof}
We now prove the following:
\begin{p4sec5}   \label{prop4s5}
$$ch^{-1}(\Box_0\mathcal{L}^H)=\psi^m(\mathcal{L}^f ),$$
where the Adams operation $\psi^m:\mathcal{K}^f\to \mathcal{K}^f$ acts on $q$ by $\psi^m(q)=q^m$.
\end{p4sec5}   
           
\begin{proof}we first show that
        \begin{align}
      \Box_{0,\mu} = m^{\text{age}(\overline{\ix}_\mu)-1/2 dim_{\mathbb{C}}\ix}\psi^m(\bigtriangleup_\mu) e^{-(\log m)c_1(T_\ix)/z}. \label{412}
       \end{align} 
       Note that $\bigtriangleup_\mu$ and $\Box_{0,\mu}$ are Euler-Maclaurin asymptotics\footnote{See (\ref{eumc}) for what this means.} of infinite products:
      \begin{align*} 
       &\bigtriangleup_\mu =\prod_{i,l} \prod^{\infty}_{r=1}s_1(x^{(l)}_i+lz/r-rz), \\
       &  \Box_{0,\mu} = \prod_{i,l}\prod^{\infty}_{r=1}s_2(x^{(l)}_i+lz/r-rz),
       \end{align*}
      where 
       \begin{align*}
      s_1(x) =\frac{x}{1-e^{-x}}  \quad \text{and} \quad s_2(x) =\frac{x}{1-e^{-mx}} = \frac{1}{m}\psi^m (s_1(x)) . 
       \end{align*} 
      It follows that 
       \begin{align*}
       \log s_2(x) = -\log(m) + \psi^m \log s_1(x) .
       \end{align*} 
      But from the definition of Euler-Maclaurin expansion we see $-\log m$ affects only the terms 
      \begin{align*}
       \sum_{i,l} \left[\int_0^{x^{(l)}_i} (- \log m) dt/z + \log m/2 \right] = \frac{(-\log m)c_1(T_\ix)}{z} - dim(\ix) \frac{\log m}{2} + \text{age}(\overline{\ix}_\mu)\log m. 
       \end{align*} 
      since the sum is taken over Chern roots of $T_\ix$ (its restriction to $\overline{\ix}_\mu$ - to be precise) (we used $B_1(x)=x-1/2$ and the definition of the age). Formula $(\ref{412})$  follows.
 
 Returning to the proof of Proposition \ref{prop4s5}, we know that:
        \begin{align}
       \bigtriangleup_\mu^{-1} ch(\mathcal{J}^\ix_f)_\mu =(J_\ix^H)_\mu . 
        \end{align}
We use the Chern character to define the Adams operation in cohomology:
         \begin{align*}
         \psi^m(a) : = ch\left( \psi^m(ch^{-1} a)\right) .
          \end{align*}
        Notice that if $a$ is homogeneous then $\psi^m (a)  = m^{deg(a)/2}a$. 
        
The $J$-function $J_\ix^H$ has degree two with respect to the grading $deg( z )=2$, $deg (Q^d )= 2\int_{d}c_1(T_\ix)$, and the age grading in Chen-Ruan orbifold cohomology. Therefore if we write $(J_\ix^H)_\mu =-z\sum_d (J_d)_\mu Q^d$, then $deg (J_d)_\mu = - deg (Q^d)-2\text{age}(\overline{\ix}_\mu)$. Hence:
        \begin{align}
     \psi^m(J_\ix^H)_\mu = \sum_d m^{c- \int_{d}c_1(T_\ix)} (-z) (J_d)_\mu Q^d.
        \end{align}  
    where $c=1-2\text{age}(\overline{\ix}_\mu)$. We can rewrite this as:
       \begin{align}
      m^{-c}\psi^m(J_\ix^H) =\sum_d e^{-\log (m) \int_{d}c_1(T_\ix)} (-z) (J_d)_\mu Q^d.  
       \end{align}   
      We now use the divisor equation to write the RHS of the above as:
\begin{align} \label{413}
\sum_d e^{-\log (m) c_1(T_\ix)/z}(-z)(J_d)_\mu Q^d = e^{-\log (m) c_1(T_\ix)/z}(J_\ix^H)_\mu.
\end{align} 
We now combine (\ref{412}) and (\ref{413}) to write:
\begin{align*}
       \Box_{0,\mu} (J_\ix^H)_\mu &= m^{\text{age}(\overline{\ix}_\mu)-1/2 dim_{\mathbb{C}}\ix}\psi^m(\bigtriangleup_\mu) e^{-(\log m)c_1(T_\ix)/z} (J_\ix^H)_\mu \\
      &= m^{\text{age}(\overline{\ix}_\mu)-1/2 dim_{\mathbb{C}}\ix}\psi^m(\bigtriangleup_\mu)m^{-c}\psi^m(J_\ix^H)_\mu\\
      & = m^{c_0}\psi^m(\bigtriangleup_\mu (J_\ix^H)_\mu ),
\end{align*}
where $c_0:= \text{age}(\overline{\ix}_\mu)-\frac{1}{2} dim_{\mathbb{C}}\ix-c$. This proves the proposition because the range of $c'\bigtriangleup J_\ix^H$ is the cone $ch(\mathcal{L}^f )$, for any scalar $c'\in \mathbb{C}$. 
\end{proof}
         
We introduce one more generating series: 
\begin{align*}
       &\delta\mathcal{J}^{st}(\delta\mathbf{t},\mathbf{T})\\
       = &\delta\mathbf{t}(q^{1/m}) 
         + \sum_a\phi_a\sum_{n,d}\frac{Q^{d}}{n!}\la \frac{\phi^a}{1-q^{1/m}L^{1/m}},\mathbf{T}(L),\ldots ,\mathbf{T}(L) ,\delta\mathbf{t}(L^{\frac{1}{m}});Tr(\Lambda^* N_{0,n,d})\ra^{\ix/\mathbb{Z}_m,f}_{0,n+2,d}.
       \end{align*} 
       Notice that if we identify $\mathcal{K}^f$ with $\mathcal{K}^\zeta$ and do the change of variables $Q^d\mapsto Q^{md}$ we obtain the localization $\mathcal{J}_\zeta(q^{1/m}\zeta^{-1})$. 

\begin{p6sec5} \label{prop6s5}
The element $ch\left(\delta\mathcal{J}^{st}(\delta\mathbf{t},\mathbf{T}) \right)$ lies in the subspace $$\Box_1\Box_0^{-1}\mathcal{T}_{J^{tw}_{\ix/\mathbb{Z}_m}}\Box_0 \mathcal{L}^H,$$
       where the input $\mathbf{T}$ is related to the application point $J^{tw}_{\ix/\mathbb{Z}_m}$ by the projection $[\ldots]_+$ along the polarisation pertaining to the identity sector in:
       \begin{align*}
     ch[1-q^m + \mathbf{T}(q)]  =[ J^{tw}_{\ix/\mathbb{Z}_m}]_+ . 
       \end{align*}
         \end{p6sec5}
       \begin{proof} 
       We use the results of \cite{to1} which explain how the twisting by characteristic classes of bundles $\pi_*(L^{-1}-1)$ and $\pi_*i_*\mathcal{O}_\mathcal{Z}$ (they were called twistings of type $\mathcal{B}$ and $\mathcal{C}$ respectively) affect the cone -- namely by a change of dilaton shift and polarization. 
       
According to the description of the virtual normal bundle $N_{0,n,d}$,  $ch(\delta\mathcal{J}^{st})$ is obtained from $\delta J^{tw}_{\ix/\mathbb{Z}_m}$ by twisting of type $\mathcal{B}$ and $\mathcal{C}$ classes of Corollaries 6.2 and 6.3 in \cite{to1}. Therefore it lies in the same space  as $\delta J^{tw}_{\ix/\mathbb{Z}_m}$, which according to the Proposition $\ref{prop4s5}$ is $\Box_1\Box_0^{-1}\mathcal{T}_{J^{tw}_{\ix/\mathbb{Z}_m}}\Box_0 \mathcal{L}^H$.
             
However, the dilaton shift (see Corollary 6.2 of \cite{to1}) changes from $-z$ to $1-e^{mz}$, and so does the space $\mathcal{H}_-$ of the polarization. Changing the input at the first marked point from $\widetilde{\phi^a}/(-z-\ops) = \phi^a/(-z/m-\ops/m)$ to $\phi^a/(1-e^{(z+\ops)/m})$ is equivalent to considering the generating series with respect to the polarization pertaining to the sector $\xi \in \mathbb{Z}_m$. The input $\mathbf{T}$ is related to $ J^{tw}_{\ix/\mathbb{Z}_m}$ by:
              \begin{align}
              ch[\mathbf{T}(q)]_+  =[ J^{tw}_{\ix/\mathbb{Z}_m}]_+ -1+e^{mz}\label{511}
              \end{align}
             due to the new polarization and dilaton shift. 
\end{proof}             
        It remains to identify the space obtained from $\mathcal{T}_{J^{tw}_{\ix/\mathbb{Z}_m}}\Box_0 \mathcal{L}^H$, after the change of variables $Q^d\mapsto Q^{dm}$, with the $\mathcal{T}_m$ in the statement of the Theorem \ref{main1}. 
        
        According to Proposition \ref{prop4s5} there exist a point $\mathcal{J}_f (\widetilde{\mathbf{T}} ) \in \mathcal{L}^f$ such that  
      $ \psi^m \mathcal{J}_f (\widetilde{\mathbf{T}})= J^{tw}_{\ix/\mathbb{Z}_m}(\mathbf{T})$.  
\begin{p7sec5} \label{prop7s5}
 The inputs $\widetilde{\mathbf{T}}, \mathbf{T}$ are related by  $\mathbf{T}=\psi^m(\widetilde{\mathbf{T}})$.   
\end{p7sec5} 

\begin{proof}

          Recall that $J_{\ix/\mathbb{Z}_m}^{tw}$ is a point on the identity sector of the twisted theory: it lies on the cone  $\Box_1 \mathcal{L}^H$ , with the corresponding dilaton shift $1-q^m$ and polarization whose negative space is spanned by \{$\frac{q^{mi}}{(1-q^m)^{i+1}}\}_{i\geq 0} =\psi^k (\mathcal{K}_-^{f})$. Then
           \begin{align*}
          &  \mathcal{J}_{f}(\widetilde{\mathbf{T}}) = (1-q) + \widetilde{\mathbf{T}} +\sum \frac{Q^d}{n!}\Phi_a\la \frac{\Phi^a}{1-qL},\widetilde{\mathbf{T}}(L),\ldots ,\widetilde{\mathbf{T}}(L)\ra_{0,n+1,d}^{f} , \\
          & J_{\ix/\mathbb{Z}_m}^{tw}(\mathbf{T}) = (1-q^m) + \mathbf{T} +\sum \frac{Q^d}{n!}\Phi_a\la\frac{\Phi^a}{1-q^mL^m},\mathbf{T}(L),\ldots ,\mathbf{T}(L);\Theta_{0,n+1,d}\ra^{\ix/\mathbb{Z}_m}_{0,n+1,d} ,
           \end{align*}
          
    and using  $ \psi^m \mathcal{J}_f (\widetilde{\mathbf{T}})= J^{tw}_{\ix/\mathbb{Z}_m}(\mathbf{T})$ it follows that  $\mathbf{T}=\psi^m(\widetilde{\mathbf{T}})$.     
   \end{proof}
   
   Moreover if we differentiate the relation $\psi^m (\mathcal{J}_{f})=J_{\ix/\mathbb{Z}_m}^{tw}$ we get
    \begin{align*}
    & \psi^m \left(\mathbf{f}(q) +\sum \frac{Q^d}{n!}\Phi_a\la \frac{\Phi^a}{1-qL},\widetilde{\mathbf{T}}(L),\ldots ,\widetilde{\mathbf{T}}(L),  \mathbf{f}(L)\ra_{0,n+2,d}^{f}\right) = \\
    & \psi^m \mathbf{f}(q) + \sum \frac{Q^d}{n!}\Phi_a\la\frac{\Phi^a}{1-q^mL^m},\mathbf{T}(L),\ldots ,\mathbf{T}(L), \psi^m \mathbf{f}(L);\theta_{0,n+2,d}\ra^{\ix/\mathbb{Z}_m}_{0,n+2,d}
    \end{align*}
  On the RHS we have a point in the tangent space $\mathcal{T}_{J_{\ix/\mathbb{Z}_m}^{tw}}\Box_1 \mathcal{L}^H$ (in the direction of $\psi^m \mathbf{f}(q))$. But if we describe the tangent space to the cone $\mathcal{L}^f$ as the image of a map $T(q,Q):\mathcal{K}_+^f\to \mathcal{K}$,
  then the LHS is $\psi^m [T(q, Q) \mathbf{f}(q)]$ which almost coincides with $\mathcal{T}_m$ defined in Definition \ref{deft}: we also need to 
   change $Q^d\mapsto Q^{dm}$ in $T$ (but not in $\mathbf{f}(q)$) because the degrees in $\mathcal{J}_\eta^{tw}$ are multiplied by $m$.  We note that the constraints in KRR force the application point $\widetilde{\mathbf{T}}$ to be the arm at $\mathbf{t}(q)=0$ i.e. $\mathcal{J}_1(0)$.            
 
   This concludes the proof of the third part of the ``only if'' implication of Theorem \ref{main1}.
             
     For the ``if'' implication of Theorem \ref{main1}, it is enough to prove that given a point in $\K$ subject to the three constraints in the statement of Theorem \ref{main1} one can uniquely reconstruct $\mathcal{J}(\mathbf{t})$ from projections to the spaces $\K_+^\zeta$. We prove that     
\begin{p8sec5}\label{p8sec5l}
The $J$-function is determined from head and stem correlators.
\end{p8sec5}  
\begin{proof} We use Propositions \ref{p1sec5l} and \ref{p3sec5l} and Lemma \ref{l1sec5l} to reconstruct the values of the $J$-function recursively on degrees $d$. First one sees that we can recover $\mathcal{J}(0)$ up to degree $d$ from head and stem correlators assuming arms and tails are known in degree strictly less than $d$. There are a few cases that require attention : for instance the head can have degree $0$, but then the stability condition implies there are at least $2$ arms - hence each has degree strictly less than $d$.  
 We can now recover the arm and tail at $\mathbf{t}=0$ and degree $d$ by projections on $\mathcal{K}_+^f$ and $\mathcal{K}^\zeta_+$ and proceed inductively to higher degree. 
      
      We can thus reconstruct the arm $\widetilde{\mathbf{T}}$ in all degrees, which also gives us the leg $\psi^m(\widetilde{\mathbf{T}})$. Now starting with any input $\mathbf{t}$ we can determine $\widetilde{\mathbf{t}}$ up to degree $d$ (assuming we know the tails in degree $<d$) from stem correlators and then use this to recover $\mathcal{J}(\mathbf{t})$ - hence the arms and tails - up to degree $d$. 
      \end{proof}

\appendix

\section{Tautological equations in genus $0$}\label{append_taut_eqns}

The purpose of this appendix is to explain several tautological equations of  K-theoretic Gromov-Witten invariants of a stack $\mathcal{X}$. We restrict our attention to genus $0$ invariants

Let $q_1,...,q_n$ be formal variables.
\begin{theorem}[string equation]
\begin{equation}
\pi_*\left(\mathcal{O}^{vir}_{\overline{\mathcal{K}}_{g,n+1}(\mathcal{X}, d)'}\left(\prod_{i=1}^n\frac{1}{1-q_iL_i}\right)\right)=\left(1+\sum_{i=1}^n\frac{q_i}{1-q_i} \right) \left( \mathcal{O}^{vir}_{\overline{\mathcal{K}}_{g,n}(\mathcal{X}, d)}\left(\prod_{i=1}^n\frac{1}{1-q_iL_i} \right)\right). \label{string}
\end{equation}
\end{theorem}

\begin{theorem}[dilaton equation]
\begin{equation}
\pi_*\left(\mathcal{O}^{vir}_{\overline{\mathcal{K}}_{g,n+1}(\mathcal{X}, d)'}\left(\prod_{i=1}^{n}\frac{1}{1-q_iL_i}\right)L_{n+1}\right)=\mathcal{O}^{vir}_{\overline{\mathcal{K}}_{g,n}(\mathcal{X}, d)}\left(\sum_{i=1}^n \frac{1}{1-q_i}\prod_{i=1}^nL_i^{-1} \right)\left(\prod_{i=1}^{n}\frac{1}{1-q_iL_i}\right).
\end{equation}
\end{theorem}

The proofs of these two Theorems are completely analogous to their counterparts in the manifold case (see \cite[Sections 4.4 and 4.5]{ypl}), given that the bundles $L_i\to \overline{\mathcal{K}}_{g,n}(\mathcal{X},d)$ are the pull-backs of the corresponding line bundles on $\overline{\mathcal{M}}_{g,n}(X, d)$ via the natural map $\overline{\mathcal{K}}_{g,n}(\mathcal{X},d)\to\overline{\mathcal{M}}_{g,n}(X, d)$, given by associating to an orbifold stable map $\mathcal{C}\to \mathcal{X}$ the induced map $C\to X$ between the coarse moduli spaces.

Let $\{e_i \}\subset K^0(\overline{I\mathcal{X}})\otimes \mathbb{Q}$ be an additive basis. Let $\{t_i\}$ be coordinates associated to this basis. Put $t:=\sum_i t_i e_i\in K^0(\overline{I\mathcal{X}})\otimes \mathbb{Q}$. Consider the following generating function of genus $0$ K-theoretic Gromov-Witten invariants without descendants:
\begin{equation}
G(t, Q):=\frac{1}{2}(t,t)+\sum_{n\geq 0} \sum_{d\in H_2(X, \mathbb{Q})^{\text{eff}}}\frac{Q^d}{n!}\langle t,...,t \rangle_{0,n,d}.
\end{equation}
Consider the following metric 
\begin{equation}
((e_i,e_j)):=G_{ij}:=\frac{\partial}{\partial t_i}\frac{\partial}{\partial t_j}G(t,Q).
\end{equation}
Note that $G_{ij}|_{Q=0}=g_{ij}:=(e_i,e_j)$. Define $G^{ij}$ to be entries of the inverse matrix of $(G_{ij})$. 

Define the quantum product on $K^0(\overline{I\mathcal{X}})$ to be 
\begin{equation}
((e_i\star e_j, e_k)):=G_{ijk}:=\frac{\partial}{\partial t_i}\frac{\partial}{\partial t_j}\frac{\partial}{\partial t_k}G(t,Q).
\end{equation}

\begin{theorem}[WDVV equation]
\begin{equation}
\sum_{\mu, \nu} G_{ij\mu}G^{\mu\nu}G_{\nu kl}=\sum_{\mu, \nu} G_{ik\mu}G^{\mu\nu}G_{\nu jl}=\sum_{\mu,\nu}G_{il\mu}G^{\mu\nu}G_{\nu jk}.
\end{equation}
\end{theorem}
\begin{proof}
The proof is completely analogous to its counterpart in the manifold case, see \cite{givental_wdvv} and \cite[Section 5.1]{ypl}. The necessary splitting properties of the virtual structure sheaves can be proved in exactly the same way as its counterpart in the manifold case (see \cite[Section 3.7]{ypl}). 
\end{proof}

\begin{corollary}
The quantum product $\star$ is associative and commutative.
\end{corollary}

We now discuss {\em topological recursion relations} in genus $0$. Consider the following generating function of genus $0$ K-theoretic descendant Gromov-Witten invariants:
\begin{equation}
\langle\langle E_1(L-1)^{k_1},...,E_n L^{k_n} \rangle\rangle_0:=\sum_{k\geq 0} \sum_d \frac{Q^d}{k!}\langle E_1(L-1)^{k_1},...,E_n L^{k_n}, t,...,t\rangle_{0, n+k, d}.
\end{equation}

\begin{theorem}[topological recursion relations]
\begin{equation}
\langle\langle e_i(L-1)^{k_1+1}, e_j L^{k_1} ,e_k L^{k_3} \rangle\rangle_0=\sum_{\mu, \nu} \langle \langle  e_i L (L-1)^{k_1}, e_\mu\rangle\rangle_0G^{\mu\nu} \langle \langle e_\nu, e_j L^{k_1} ,e_k L^{k_3}\rangle \rangle_0.
\end{equation}
\end{theorem}
\begin{proof}
The proof uses comparison formula (\ref{05a2bis}), where $\ol_1 =1 \in K^0(\overline{M}_{0,3})$. Notice that both hand sides of the equality 
contain the same power of $L_1$, so in general these relations are not obviously recursions.  
\end{proof}

\section{Ancestors and descendants}

  In this appendix we prove the relation between the ancestor and descendant potentials, which does not appear anywhere in the literature in the K-theoretic setting. We first define the main objects of interest. Let $ft_{n,l}$ denote the composition $$ft_{n,l}: \overline{\mathcal{K}}_{0,n+l}(\ix,d)\to \overline{\mathcal{M}}_{0,n+l}(X,d)\to\overline{M}_{0,n}$$ where the second map forgets the last $l$ marked points. Denote by $\overline{L}_i:=ft_{n,l}^*(L_i)$, $i=1,..,n$. Let $\tau \in K^0(\overline{I\ix})$. The {\em genus $0$ ancestor potential} is defined as
   \begin{align*}
  \overline{\mathcal{F}}_\tau^0=\sum_{l,n,d}\frac{Q^d}{n!l!}\la \mathbf{t}(\ol),\ldots ,\mathbf{t}(\ol),\tau,\ldots \tau\ra_{0,n+l,d}
   \end{align*}
where $\tau$ occurs  in the last $l$ entries. Its differential, the ancestor $J$-function, gives rise to a Lagrangian space $$\mathcal{L}_\tau\subset \mathcal{K}_\tau,$$ where $\mathcal{K}_\tau$ is the loop space defined to be the same as $\mathcal{K}$ but with the symplectic form based on the nonconstant pairing $G_{\alpha\beta}$ defined in the Appendix \ref{append_taut_eqns}. 
  
   We introduce the notation
   \begin{align*}
   \la E_1L^{k_1},\ldots , E_nL^{k_n} \ra_{0,n}(\tau):=\sum_{m,d}\frac{Q^d}{m!}\la E_1L^{k_1},\ldots , E_n L^{k_n},\tau,\ldots ,\tau \ra_{0,n+m,d}.
   \end{align*}
  Define the $S_\tau$ matrix by
   \begin{align*}
  S_{\alpha\beta}(q,\tau):= (\phi_\beta,\phi_\alpha)+ \left\langle\frac{\phi_\beta}{1-qL},\phi_\alpha\right\rangle_{0,2}(\tau).
   \end{align*}
   and 
   \begin{align}
  S\phi_\beta := \sum_\alpha \phi_\alpha G^{\alpha\mu}(\tau) S_{\mu\beta}(q^{-1},\tau). 
   \end{align}

 First we prove that $S_\tau$ is a symplectomorphism\footnote{For simplicity we omit the subscript indicating the dependence of $S$ on $\tau$.}. The condition on $S$ being symplectic transformation reads
  \begin{align}
   S^*(q^{-1}) S(q)= I \label{002a}
  \end{align}
  where $I$ is the identity matrix. We show that:
  \begin{align}\label{eqn_S_sympl}
S_{\mu\alpha}(q_1)G^{\mu\nu}S_{\nu\beta}(q_2)=(1-q_1 q_2)\la\frac{\phi_\beta}{1-q_2L},\frac{\phi_\alpha}{1-q_1L}\ra_{0,3}(\tau)+g_{\alpha\beta}.
  \end{align}
 Formula (\ref{002a}) follows by setting $q_2=q_1^{-1}$.  The string equation (\ref{string}) shows that
 \begin{align*}
 S_{\mu\alpha}(q_1)=(1-q_1)\la \frac{\phi_\alpha}{1-q_1L},1,\phi_\mu\ra_{0,3}(\tau).
 \end{align*}
 Hence:
 \begin{align}
 S_{\mu\alpha}(q_1)G^{\mu\nu}S_{\nu\beta}(q_2)= (1-q_1)(1-q_2)\la \frac{\phi_\alpha}{1-q_1L},1,\phi_\mu\ra_{0,3}(\tau)G^{\mu\nu}(\tau)\la \frac{\phi_\beta}{1-q_2L},1,\phi_\nu\ra_{0,3}(\tau).\label{01}
 \end{align}
 
 We can now use WDVV relation\footnote{In the more general form that involves descendant line bundles, whose proof is the same.}  in Appendix A to swap some inputs, formula (\ref{01}) becomes:
 
 \begin{align}
 (1-q_1)(1-q_2)\la \frac{\phi_\alpha}{1-q_1L},\frac{\phi_\beta}{1-q_2L},\phi_\mu\ra_{0,3}(\tau)G^{\mu\nu}(\tau)\la\phi_\nu, 1,1\ra_{0,3}(\tau).\label{02}
 \end{align}
 By string equation and definition, the last two factors give $\delta_{\mu 1}$, so the quantity in formula (\ref{02}) is
 \begin{align}
 (1-q_1)(1-q_2)\la \frac{\phi_\alpha}{1-q_1L},\frac{\phi_\beta}{1-q_2L},1\ra_{0,3}(\tau).\label{03}
 \end{align}
Using again the string equation, (\ref{03}) becomes 
 \begin{align*}
 (1-q_1)(1-q_2)\left[(1+\frac{q_1}{1-q_1}+\frac{q_2}{1-q_2})\la \frac{\phi_\alpha}{1-q_1L},\frac{\phi_\beta}{1-q_2L}\ra_{0,2}(\tau)+\la \frac{\phi_\alpha}{1-q_1L},\frac{\phi_\beta}{1-q_2L},1\ra_{0,3,0}\right]
 \end{align*}
 This proves the claim (\ref{eqn_S_sympl}).
 \begin{tha02} \label{th2a}
 Let $\mathcal{L}$ be the cone of quantum K-theory of $\ix$. Then
  \begin{align*}
  \mathcal{L}_\tau =S_\tau \mathcal{L}.
  \end{align*}
 \end{tha02}
 To compare ancestor and descendant classes, let $\ol_1$ be the pull-back along the map $ft_{n,l}$  to $\overline{\mathcal{K}}_{0,n+l}(\ix,d)$ and let $D$ be the divisor which parametrizes maps such that the component on which marked the point $1$ lies gets contracted by the forgetful morphism. It is known that 
 \begin{align}
 L_1 = \ol_1 \otimes \mathcal{O}(D).\label{05a2}
 \end{align}
 
 This gives  
 $$ L_1 - \ol_1 = \ol_1 \otimes (\mathcal{O}(D)-1).$$
 
 Using the exact sequence 
 $$ 0\to \mathcal{O}\to \mathcal{O}(D)\to \mathcal{O}(D)_{\vert D}\to 0,$$
 we find that $\mathcal{O}(D)-1 = \mathcal{O}(D)_{\vert D}$, which is the normal bundle to $D$. This is also identified with $Hom(\ol_1, L_1)_{\vert D} = (L_1 \otimes \ol^\vee_1)_{\vert D}$.  Then the comparison formula takes the form
 \begin{align}
  L_1 -\ol_1 = L_1\otimes \mathcal{O}_D. \label{05a2bis}
 \end{align} 
 $D$ is not irreducible, but a divisor with normal crossings $\cup_i D_i$. Its structure sheaf can be expressed in terms of the structure sheaves of its components as
  \begin{align*}
 \sum_i \mathcal{O}_{D_i} - \sum_{i<j}\mathcal{O}_{D_i \cap D_j} + \sum_{i<j<k}\mathcal{O}_{D_i\cap D_j \cap D_k}-\ldots .
  \end{align*}
   For more on this see \cite{givental_wdvv}. In the following computation it is convenient to 
   write the input 
   \begin{align*}
    \mathbf{t}(L) = \sum^n_{k=0} t_k (L-1)^k.
   \end{align*}  
    This is possible because the line bundles $L_i$ have a  minimal polynomial $P(L)=0$, $P(0)\neq 0$. 
   
   We rewrite relation $(\ref{05a2bis})$  as $L_1 - 1= \ol_1 -1+ L_1 \otimes \mathcal{O}_D$ and use it to decrease the power of $L_1 -1$ in correlators:
 \begin{align*}
 \la (L-1)^a (\ol-1)^b,\ldots\ra_{0,n}(\tau)  =&\la (L-1)^{a-1}(\ol-1)^{b+1},\ldots \ra_{0,n}(\tau)\\
   &+  \la L(L-1)^{a-1},\phi_\mu\ra_{0,2}(\tau)G^{\mu\nu}(\tau)\la \phi_\nu (\ol-1)^{b},\ldots\ra_{0,n}(\tau).
 \end{align*}
 Applying this repeatedly we get :
 \begin{align*}
 \la \mathbf{t}(L),\ldots \ra_{0,n}(\tau)= &\la t_0,\ldots\ra_{0,n}(\tau) \\
     &+ \la t_1(\ol-1),\ldots \ra_{0,n}(\tau)+ \la t_1 L,\phi_\mu\ra_{0,2}(\tau)G^{\mu\nu}(\tau)\la \phi_\nu,\ldots \ra_{0,n}(\tau) \\
     &+ \la t_2(L-1)(\ol-1),\ldots \ra_{0,n}(\tau)+\la t_2L(L-1),\phi_\mu\ra_{0,2}(\tau)G^{\mu\nu}(\tau)\la \phi_\nu,\ldots \ra_{0,n}(\tau)+\ldots
 \end{align*}
 Rearranging the above sum after powers of $(\ol-1)$, we see that the coefficient of $(\ol -1)^i$ equals
 \begin{align}
 t_i + \la t_{i+1}L,\phi_\mu\ra_{0,2}(\tau)G^{\mu\nu}\phi_\nu +\la t_{i+2}L(L-1),\phi_\mu\ra_{0,2}(\tau)G^{\mu\nu}\phi_\nu+\ldots .
 \end{align} 
 This is the same as the coefficient of $(q-1)^i$ in the $[S_\tau \mathbf{t}(q)]$ as one sees by expanding
 \begin{align*}
 S_{\alpha\beta}(q^{-1},\tau)= (\phi_\beta,\phi_\alpha)+ \left\langle \frac{\phi_\beta}{1-q^{-1}L},\phi_\alpha\right\rangle_{0,2}(\tau) .
 \end{align*} 
 So in the end if we denote
  $\overline{\mathbf{t}}: = [S_\tau \mathbf{t}]_+$  the power series truncation of $S_\tau \mathbf{t}$ we have 
 \begin{align*}
  \la \mathbf{t}(L),\ldots \ra_{0,n}(\tau)= \la \overline{\mathbf{t}}(\ol),\ldots \ra_{0,n}(\tau).
\end{align*} 
 Of course the same procedure can be applied at  all marked points to get
 \begin{align*}
  \la \mathbf{t}(L),\ldots ,\mathbf{t}(L)\ra_{0,n}(\tau)= \la \overline{\mathbf{t}}(\ol),\ldots ,\overline{\mathbf{t}}(\ol)\ra_{0,n}(\tau).
 \end{align*} 
  Taking into account the dilaton shift, if we set
 $\overline{\mathbf{q}}=[S_\tau \mathbf{q}]_+$  we get 
 \begin{align}
 \overline{\mathbf{t}} = [S_\tau \mathbf{t}]_+ + [S_\tau(1-q)]_+ +q-1.
 \end{align}
We claim that $[S(1-q)]_+ = 1-q - \tau $. By definition the LHS is
 \begin{align}
 \sum_\alpha \phi_\alpha G^{\mu\alpha}(\tau) [S_{\mu 1}(q^{-1},\tau)(1-q)]_+.\label{09a2}
 \end{align}
 Notice that
 \begin{align*}
 \left[ \frac{ (1-q)}{1-q^{-1} L}\right]_+ = (1-q) - L.
 \end{align*}
 This gives 
 \begin{align}
[S_{\mu 1}(q^{-1},\tau)(1-q)]_+ = (1-q,\phi_\mu) +\la 1-q-L,\phi_\mu\ra_{0,2}(\tau)= (1-q)G_{1\mu}(\tau) -G_{\tau\mu}(\tau). \label{10a2}
 \end{align}
We have used  a version of the dilaton equation to get rid of the input $L$: 
\begin{align*}
 \la -L,\phi_\mu\ra_{0,2}(\tau) = -\la\phi_\mu\ra_{0,1}(\tau).
 \end{align*}
 Plugging (\ref{10a2}) into (\ref{09a2}) proves the claim. It follows that
 \begin{align*}
 \overline{\mathbf{t}} = [S_\tau \mathbf{t}]_+ -\tau. 
 \end{align*}
 
 For a fixed $n_0$ we have
  \begin{align*}
  \sum_{n\geq 0}\frac{1}{(n+n_0)!}\la \mathbf{t}(L),\ldots ,\mathbf{t}(L)\ra_{0,n+n_0}(0)=\sum_{n\geq 0}\frac{1}{(n+n_0)!}\la \mathbf{t}(L)-\tau,\ldots ,\mathbf{t}(L)-\tau \ra_{0,n+n_0}(\tau) .
  \end{align*}
  The above computation then shows that 
  \begin{align*}
    \sum_{n\geq 0}\frac{1}{(n+n_0)!}\la \mathbf{q}(L),\ldots ,\mathbf{q}(L)\ra_{0,n+n_0}(0)=\sum_{n\geq 0}\frac{1}{(n+n_0)!}\la \overline{\mathbf{q}}(\ol),\ldots ,\overline{\mathbf{q}}(\ol) \ra_{0,n+n_0}(\tau) .
    \end{align*}
To prove a relation on potentials we need to take into account the terms missing from the ancestor potential, which correspond to $n_0\leq 2$.
 A computation using dilaton equation shows that
\begin{align*}
\frac{1}{2}\la \mathbf{t}(L)-\tau,\mathbf{t}(L)-\tau \ra_{0,2}(\tau) + \la \mathbf{t}(L)-\tau \ra_{0,1}(\tau) + 
\la  \ra_{0,0}(\tau) = \frac{1}{2}\la \mathbf{t}(L)+1-L,\mathbf{t}(L)+1-L\ra_{0,2}(\tau) .
\end{align*} 
   Wrapping up, we get
   \begin{align}
      \mathcal{F}_0 (\mathbf{q}) =\overline{\mathcal{F}}^0_\tau([S\mathbf{q}]_+) +\frac{1}{2}\la \mathbf{q}, \mathbf{q}\ra_{0,2}(\tau).
      \end{align}
 Notice that $[S\mathbf{q}]_+$ is the $pq$ part of the Hamiltonian of $S$ and that $1/2 \langle \mathbf{q}, \mathbf{q}\rangle_{0,2}$. is the $q^2$ part (there is no $p^2$ part). We refer to \cite{gi} for more on quadratic hamiltonians. The statement of  Theorem \ref{th2a} now follows from the Hamilton-Jacobi equation. We refer to \cite[Section 9.1]{mcduff2}, leaving the details to the reader.


\begin{thebibliography}{plain}

\bibitem{acv}
D. Abramovich, A. Corti, A. Vistoli, {\em Twisted bundles and admissible covers}, Comm. Algebra \textbf{31} (2003), 3547--3618.  
  
 \bibitem{abgrvi}
 D. Abramovich, T. Graber, A. Vistoli, {\em Algebraic orbifold quantum products}, in ``Orbifolds in Mathematics and Physics (Madison, WI, 2001)'', 1--24. Contemp. Math., \textbf{310}, Amer. Math. Soc., Providence, RI, 2002.  
 
 \bibitem{abgrvi2}
 D. Abramovich, T. Graber , A. Vistoli, {\em Gromov-Witten theory of Deligne-Mumford stacks}, Amer. J. Math. \textbf{130} (2008), 1337--1398. 
 
 \bibitem{abvi}
 D. Abramovich, A. Vistoli, {\em Compactifying the space of stable maps}, J. Amer. Math. Soc. \textbf{15}, no.1 (2002), 27--75.
 
 \bibitem{av_notes}
 D. Abramovich, A. Vistoli, {\em Twisted Stable Maps and Quantum Cohomology of Stacks}, in ``Intersection Theory and Moduli'', ICTP Lecture Notes Series, vol. 19, 2004.
  
 \bibitem{behfan}
 K. Behrend, B. Fantechi, {\em The intrinsic normal cone}, Invent. Math. \textbf{128} (1997), 45--88.
 
 \bibitem{B_F}
A. Braverman, M. Finkelberg, {\em Semi-infinite Schubert varieties and quantum K-theory of flag manifolds}, J. Amer. Math. Soc. \textbf{27}, no. 4 (2014), 1147--1168  .
  
 \bibitem{chru2}
  W. Chen, Y. Ruan, {\em A New Cohomology theory for Orbifold}, Comm. Math. Phys., \textbf{248}, no.1 (2004), 1--31.
  
  \bibitem{cfk}
I. Ciocan-Fontanine, M. Kapranov, {\em Virtual fundamental classes via dg-manifolds}, Geom. Top. \textbf{13} (2009), 1779--1804.
  
  \bibitem{tom}
 T. Coates, {\em Riemann-Roch theorems in Gromov-Witten theory}, Ph.D. Thesis, 2003.
 
 
 
 \bibitem{coatesgiv}
 T. Coates, A. Givental, {\em Quantum Riemann-Roch, Lefschetz and Serre}, Ann. Math \textbf{165}, no.1 (2007), 15--53.
\bibitem{mcduff2}
 D. McDuff, D. Salamon, {\em Introduction to symplectic topology}. Second Edition. The Clarendon Press, Oxford University Press, New York, 1998.
   
 \bibitem{fg} 
  B. Fantechi, L. Goettsche, {\em Riemann-Roch theorems and elliptic genus for virtually smooth Schemes}, Geom. Top. \textbf{14} (2010), 83--115.
  
  \bibitem{givental_wdvv} 
    A. Givental, {\em On the WDVV equation in quantum K-theory}, Dedicated to William Fulton on the occasion of his 60th birthday. Michigan Math. J. 48 (2000), 295--304.
  
    \bibitem{gi}
 A. Givental, {\em Gromov-Witten invariants and quantization of quadratic Hamiltonians}, Mosc. Math. J. \textbf{1}, no. 4 (2001), 551--568 .
 
 \bibitem{given1}
   A. Givental, {\em Symplectic geometry of Frobenius structures}, in ``Frobenius Manifolds'', 91--112, Aspects Math., E36 (2004), Vieweg, Wiesbaden.
 
 \bibitem{giv_lee}
 A. Givental, Y.-P. Lee, {\em Quantum K-theory on flag manifolds, finite difference Toda lattices and quantum groups}, Invent. Math. \textbf{151} (2003), 193--219. 
  \bibitem{gito}
  A. Givental, V. Tonita, {\em The Hirzebruch-Riemann-Roch Theorem in true genus-$0$ quantum K-theory}, in ``Symplectic, Poisson, and Noncommutative Geometry'', 43--92, Math. Sci. Res. Inst. Publications, vol. 62, Cambridge Univ. Press, 2014, arXiv:1106.3136.
    
  \bibitem{jaki}
 T. Jarvis, T. Kimura, {\em Orbifold quantum cohomology of the classifying space of a finite group}, in ``Orbifolds in mathematics and physics, (Madison WI, 2001)'', 123--134, Contemp. Math., \textbf{310}, Amer. Math. Soc., Providence, RI, 2002.  
  
 
\bibitem{kawasaki}  
T. Kawasaki, {\em The Riemann-Roch theorem for complex V-manifolds}, Osaka J. Math. \textbf{16}, no. 1 (1979), 151--159.
 
 \bibitem{ypl}
 Y.-P. Lee, {\em Quantum K-theory. I. Foundations}, Duke Math. J. \textbf{121} (2004), no. 3, 389--424.
 
 \bibitem{r} M. Robalo, personal communication.

 \bibitem{toen}
 B. To\"en, {\em Th\'eor\`emes de Riemann-Roch pour les champs de Deligne-Mumford}, K-Theory \textbf{18}, no.1 (1999), 33--76.
 
 \bibitem{to_k}
  V. Tonita, {\em A virtual Kawasaki formula}, Pacific J. Math. \textbf{268} (2014), no. 1, 249--255.
 \bibitem{to1}
  V. Tonita, {\em Twisted Gromov-Witten invariants}, Nagoya Math. J. \textbf{213} (2014), 141--187.
  
 \bibitem{tseng}
 H.-H. Tseng, {\em Orbifold quantum Riemann-Roch, Lefschetz and Serre}, Geom. Top. \textbf{14} (2010), 1--81.
 \end{thebibliography}
\end{document}